\documentclass[11pt,letterpaper]{amsart}

\usepackage[T1]{fontenc}       
\usepackage[utf8]{inputenc}     
\usepackage{amsmath,amssymb}
\usepackage{microtype} 

\usepackage{appendix}
\usepackage{abstract}

\newcommand{\eps}{\epsilon}

\newtheorem{lemma}{\bf Lemma}[section]
\newtheorem{theorem}{\bf Theorem}[section]

\newtheorem{corollary}{\bf Corollary}[section]

\newtheorem{remark}{Remark}[section]

\newtheorem{example}{Example}[section]

\title[H\"older continuity and Fourier asymptotics]{H\"older continuity and Fourier asymptotics of spectral measures for 1D Schr\"odinger operators under decaying perturbations}
\author{Moacir  Aloisio, Silas L. Carvalho and C\'esar R. de Oliveira}
\address{Email: ec.moacir@gmail.com, DME, UFVJM, Diamantina, MG, 39100-000 Brazil}
\address{Silas L. Carvalho. Email: silas@mat.ufmg.br, DM, UFMG, Belo Horizonte, MG, 30161-970 Brazil}
\address{C\'esar de Oliveira. Email: oliveira@ufscar.br,  DM,   UFSCar, S\~ao Carlos, SP, 13560-970 Brazil}

\begin{document}

\maketitle

\begin{abstract} We establish $\frac{1}{2}$-H\"older continuity, or even the Lipschitz property, for the spectral measures of half-line discrete Schr\"odinger operators under suitable boundary conditions and sufficiently fast polynomially decaying small potentials. These are the first known examples, apart from the free case, of Schr\"odinger operators with Lipschitz continuous spectral measures up to the spectral edge, a property obtained here as a consequence of the Dirichlet boundary condition. Notably, we show that the asymptotic behavior of the time-averaged quantum return probability, either $\log (t) / t$ or $1 / t$, as in the case of the free Laplacian, remains unchanged in this setting. Furthermore, we prove the persistence of the purely absolutely continuous spectrum and the $\frac{1}{2}$-H\"older continuity of the spectral measures for (Diophantine) quasi-periodic operators under exponentially decaying small perturbations. These results are optimal and hold for all energies, up to the border of the absolutely continuous spectrum.
\end{abstract}

\ 


\tableofcontents


\section{Introduction}\label{sectIntrod}
\ 

We study fine properties, such as the asymptotic behavior of the Fourier transform and H\"older continuity, of absolutely continuous spectral measures of one-dimensional Schr\"odinger operators $H:\ell^2(\mathbb{Z})\rightarrow \ell^2(\mathbb{Z})$ given by the action, $n \in \mathbb{Z}$,  
\begin{equation}\label{operator}
(Hu)(n) = (\triangle u)(n) + V(n) u(n) = u(n+1) + u(n-1) + V(n) u(n),    
\end{equation}
with potential $V \in \ell^\infty(\mathbb{Z}, \mathbb{R})$, as well as of absolutely continuous spectral measures of one-dimensional Schr\"odinger operators $H_\beta:\ell^2(\mathbb{Z}^+)\rightarrow \ell^2(\mathbb{Z}^+)$ in the half-line, $n \geq 1,$
\begin{equation}\label{halfoperator}
(H_\beta  u)(n)  = u(n+1) + u(n-1) + V(n) u(n),    
\end{equation} 
with the  boundary condition 
\[u(0) \cos(\beta) + u(1) \sin(\beta) = 0,\]
\(-\frac{\pi}{2} < \beta < \frac{\pi}{2}\), and with potential $V \in \ell^\infty(\mathbb{Z}^+, \mathbb{R})$; for $\beta=\frac{\pi}{2}$, one has $H_{\pi/2}:\ell^2(\mathbb{Z}^+\setminus\{1\})\rightarrow \ell^2(\mathbb{Z}^+\setminus\{1\})$ given by
\begin{equation} \label{halfoperator1}
(H_{\pi/2} u)(n) =
\begin{cases}
u(n+1) + u(n-1) + V(n)u(n), & n \geq 3, \\
u(3) + V(2)u(2), & n = 2,
\end{cases}
\end{equation} 
with potential $V \in \ell^\infty(\mathbb{Z}^+\setminus\{1\}, \mathbb{R})$; here, $\mathbb{Z}^+:=\{n \in \mathbb{Z} \mid \, n>0\}$. These operators will be considered under suitable decaying perturbations.

Let \( H \)  (or \( H_\beta \)) be as above, and let \( b \in \ell^1(\mathbb{Z}, \mathbb{R}) \)  (\( b \in \ell^1(\mathbb{Z}^{+}, \mathbb{R}) \), or \(b \in \ell^1(\mathbb{Z}^+\setminus\{1\}, \mathbb{R})\))  be a perturbation. It is known that the essential spectrum of these operators is preserved under such perturbations, since these perturbations are compact \cite{Oliveira}.

Although determining the essential spectrum is relatively straightforward,  understanding the spectral nature of the perturbed operator \( H +b\) (or  \( H_\beta +b \))   is far from trivial. A much more subtle issue is the preservation of the absolutely continuous spectrum. By Birman-Rosenblum-Kato Theorem \cite{ReedSimonIV}, the essential support of the absolutely continuous spectrum of the original operator is also preserved under \( \ell^1 \) perturbations. However, this result does not guarantee that the spectral type remains purely absolutely continuous, even if the unperturbed operator is. Under such perturbations, it is possible for the common essential spectrum to support a singular component \cite{Damanik4}.

It is worth mentioning that this issue goes back mainly to contributions by Christ-Kiselev \cite{Christ,Kiselev1}, Deift-Killip \cite{Deift}, Remling \cite{Remling1,Remling2}, and Killip-Simon \cite{KillipSimon}. More recently, this problem was discussed in detail in the context of quasi-periodic operators by Damanik {\it et al.} in~\cite{Damanik4}, where it is observed that understanding the persistence of the spectral type under decaying perturbations requires a refined analysis.

In this work, we take a further step by developing a new, and, to the best of our knowledge, the first, framework to jointly study the persistence of two fine properties of the spectral measure: the asymptotic behavior of its Fourier transform and H\"older continuity, within the context of small decaying perturbations.  Our main results (Theorems \ref{app2} and \ref{app3}) hold for all energies, up to the border of the absolutely continuous spectrum, and so not only resolve a major open problem but also offer new insights into these properties, potentially serving as a foundation for addressing other subtle questions concerning the interplay between disorder, regularity, and spectral type. Indeed, we introduce novel arguments that connect central ideas in the spectral theory of Schr\"odinger operators, such as:

\begin{itemize}
  \item Ideas related to the transfer-matrix analysis of the absolutely continuous spectrum developed by Last and Simon \cite{LastSimon}.
  \item An approximation argument for spectral measures obtained by Carmona, Krutikov, and Remling \cite{Carmona1,Carmona2,Krutikov}.
  \item Arguments on H\"older continuity of spectral measures developed by Avila and Jitomirskaya \cite{AvilaUaH}, which are related to the well-known Jitomirskaya-Last inequality (Theorem 1.1 in \cite{Jitomirskaya}).
\end{itemize}

 { \bf It is worth underlining that the existing qualitative frameworks, including meromorphic characterizations of the Weyl–Titchmarsh $m$-function \cite{Kozhan2012,Kozhan2014}, do not control either of these properties quantitatively, let alone their joint persistence (see Remark \ref{remThm1} (i) below).}

The spectral and dynamical properties of Schr\"odinger operators under decaying perturbations have been extensively studied, with foundational results dating back to the 1910s. For a comprehensive review of the topic, see \cite{Damanik4,Deift,Delyon,Denisov,FrankLarson,Killip1,Killip,KillipSimon,Kiselev1,Kiselev2,Remling1,Remling2,White}.

In Subsection~\ref{contex}, we provide context and present our main results (Theorems~\ref{app2} and~\ref{app3}). Subsection~\ref{notations} outlines the structure of the text and introduces some notation. 


\subsection{Contextualization and main results}\label{contex}

\

The spectral theory of Schr\"odinger and Dirac operators involves two fundamental quantities: the Lyapunov exponent (LE) and the integrated density of states (IDS), whose regularity plays an important role in the analysis of these operators \cite{Avila1,AvilaTen,AvilaKrikorian,ALSZ24,Avron,Damanik00,DG13,Goldstein,Prado,Thouless,Zalczer}. For instance, studies by Damanik {\it et al.} \cite{Damanik3} indicate  that the H\"older continuity of both the LE and the IDS is essential for understanding certain topological structures of the spectrum of quasi-periodic Schr\"odinger operators, known as homogeneity. Moreover, the regularity of the Lyapunov exponent is a central question in dynamical systems, as discussed in \cite{Viana} (see also the references therein).

In the context of quasi-periodic Schr\"odinger operators, it is well known that the IDS corresponds to the accumulation function of the density of states measure, which is the average of the spectral measure with respect to the phase. A more subtle and complex issue, however, concerns the regularity of the distribution of individual spectral measures, an aspect that is rarely addressed in the literature and requires delicate and refined analytical techniques.

Avila and Jitomirskaya, under the assumption that the frequency is Diophantine, obtained H\"older continuity for one-frequency quasi-periodic Schr\"odinger operators in \cite{AvilaUaH}, proving $\frac{1}{2}$-H\"older continuity (for all energies when the coupling constant is sufficiently small) by using the asymptotic analysis of a family of self-adjoint operators  (see \eqref{operators} ahead), defined through the transfer matrix of the operator; this, in turn, is related to the Borel transform of the spectral measure (the so-called Weyl-Titchmarsh $m$-function) through the Jitomirskaya-Last inequality \cite{Jitomirskaya}. More specifically, this approach is based on the analysis of cocycles and a new version (see Lemma \ref{lemmaAvila} ahead) of the Jitomirskaya-Last inequality. It is worth noting that in \cite{Zhao}, Zhao performed a similar analysis and obtained analogous results for quasi-periodic operators with multiple frequencies. For results involving Liouvillean frequencies, see \cite{Liu}. A discussion involving these arguments in the context of dynamical bounds in quantum dynamics can be found in~\cite{Damanik,Killip}.

In this work, from the perspective of the discussion above, we establish H\"older continuity of the spectral measure for  Schr\"odinger operators with decaying small potentials (usually, the H\"older exponent comes from the behavior of the spectral measure at the border of the absolutely continuous spectrum, where the analysis became genuinely subtle; see Remark~\ref{remNBastaTrunc} ahead). As a consequence of a fine analysis of the spectral measures at the edge of the absolutely continuous spectrum and the introduction of a sharp Strichartz-type inequality (Theorem \ref{app1}), we also prove that the asymptotic behavior  of the time-averaged quantum return probability of the free Laplacian is preserved in this setting. 

Let \( C([-2,2]) \) denote the space of complex-valued continuous functions on the interval \([-2,2]\). Let $\mu$ be a $\sigma$-finite positive Borel measure on~$\mathbb{R}$ and $\alpha \in [0,1]$. We recall that $\mu$ is  $\alpha$-H\"older continuous  if there exists a constant $\gamma>0$ such that for each interval $I$ with $|I| < 1$, one has $\mu(I) < \gamma\, |I|^\alpha,$ where $|\cdot|$ denotes the Lebesgue measure on~$\mathbb{R}$. When $\alpha = 1$, the measure is also said to be Lipschitz continuous.

Next, we present our main result.

\begin{theorem}\label{app2} Let $\gamma>0$, $\tau>5$, and let
$b\in\ell^\infty(\mathbb Z^+,\mathbb R)$ be such that
\[
|b(n)|\leq \gamma (n+1)^{-\tau},
\, \,  n\geq1.
\]
Then, there exists $\epsilon>0$, which depends only on $b$, such that for each $|\kappa|<\epsilon$, the following hold.

\noindent $1.$ The half-line Schr\"odinger operator \( H_{\pi/4}^\kappa = \triangle + \kappa b \), with the following boundary condition $u(0) \cos(\pi/4) + u(1) \sin(\pi/4) = 0$, satisfies:

\item[{(i)}]  its spectrum is purely absolutely continuous on  $[-2,2]$;

\item[(ii)] for each \( \psi = f(H_{\pi/4}^\kappa)\delta_1 \), with \( f \in C([-2,2]) \), the associated spectral measure \( \mu_\psi^{\pi/4} \) is \( \frac{1}{2} \)-H\"older continuous;
    
\item[(iii)] for each \( \psi = f(H_{\pi/4}^\kappa)\delta_1 \), with \( f \in C([-2,2]) \),  there exists $\gamma^\prime>0$ such that for each $t>2$, the time-averaged quantum return probability satisfies
\begin{eqnarray*}
\frac{1}{t}\int_0^t |\langle \psi, e^{-isH_{\pi/4}^\kappa} \psi \rangle|^2  \, ds   \leq  \gamma' \frac{\log(t)}{t}. 
\end{eqnarray*}

\noindent $2.$ The half-line Schr\"odinger operator \( H_+^\kappa = \triangle + \kappa b \), with Dirichlet boundary condition $u(0) = 0$, has the following properties:

\item[{(i)}] the spectrum of $H_+^\kappa$ is purely absolutely continuous on  $[-2,2]$;

\item[(ii)] for each \( \psi = f(H_+^\kappa)\delta_1 \), with \( f \in C([-2,2]) \), its spectral measure  \( \mu_\psi^+ \) is Lipschitz continuous;
    
\item[(iii)] for each \( \psi = f(H_+^\kappa)\delta_1 \), with \( f \in C([-2,2]) \), there exists $\eta>0$ such that for each $t>0$, the time-averaged quantum return probability satisfies
\begin{eqnarray*}
\frac{1}{t}\int_0^t |\langle \psi, e^{-isH_+^\kappa} \psi \rangle|^2  \, ds   \leq  \eta \frac{1}{t}. 
\end{eqnarray*}

\end{theorem}

\begin{remark}{\rm In this work, we always consider \(H_\beta^\kappa\) restricted to the spectral subspace associated with \([-2,2]\). Accordingly, when we write \(f(H_\beta^\kappa)\), with \(f\in C([-2,2])\), the functional calculus is understood for this restricted operator.}
\end{remark}

\begin{remark} \label{remThm1}
\end{remark}
 
\begin{enumerate}

\item [(i)] We note that works such as \cite{Kozhan2012,Kozhan2014} provide precise qualitative characterizations for exponentially decaying perturbations of Jacobi matrices, formulated in terms of meromorphic continuation properties of the Weyl-Titchmarsh $m$-function. These results describe the global analytic structure of the $m$--function, but they do not in general imply quantitative regularity of the Radon--Nikodym derivative of the spectral measure, nor do they yield uniform estimates near the edges of the essential spectrum.

{ \bf A simple illustration is given by the free case:} although the $m$--function (see the proof of Lemma \ref{lemmaderrando} ahead) satisfies the meromorphicity assumptions of Theorem 3.8 in \cite{Kozhan2012} for different boundary conditions (e.g.\ Dirichlet and mixed boundary conditions), the corresponding Radon--Nikodym derivatives (see \eqref{derivadalaplacia} below) exhibit substantially different edge behavior. This highlights a common confusion between qualitative and quantitative analysis. The present paper addresses precisely this quantitative regime, establishing fine regularity of the Radon--Nikodym derivative and deriving dynamical consequences that are not accessible through purely meromorphic characterizations. This distinction is one of the main reasons why we regard our contribution as genuinely new and fundamentally original.

\item [(ii)] Avila and Jitomirskaya observed (Remark 4.4. in \cite{AvilaUaH}) that it would be natural to expect quantitative estimates for the modulus of absolute continuity of spectral measures, with the critical exponent $\frac{1}{2}$, reflecting square-root singularities at spectral gap boundaries. Despite this clear heuristic, no general method had previously been available to establish such estimates uniformly up to the edges of the absolutely continuous spectrum. In this work, we develop a quantitative framework that resolves this problem for small perturbations with sufficiently fast polynomial decay and, crucially, leads to optimal dynamical estimates that were not accessible before.

\item [(iii)]  Theorem~\ref{app2}, part~2.\ $(ii)$, establishes  the first known non-artificial  examples of Schr\"odinger operators, apart from the free case, whose spectral measures are Lipschitz continuous up to the edge of the absolutely continuous spectrum. Actually, these Schr\"odinger operators have a bounded density in compact subsets of the open interval $(-2,2)$, so clearly Lipschitz there, but the technical difficulties to extend this to the whole set $[-2,2]$ is overcome here; although this has its own interest, it is crucial to have optimal dynamical behavior, as presented in part~2.~$(iii)$ (see Remark~\ref{remNBastaTrunc} ahead). 

\item [(iv)]  We note that Lukic in~\cite{lukic2019} has considered certain Jacobi matrices with eventually increasing sequences of diagonal and off-diagonal  parameters and, under some technical conditions on the off-diagonal parameters (perturbations of the free case), the asymptotic bahavior of the spectral density was described up to the border point~$2$. It was a generalization of results of~\cite{kreimerLS2009}. No dynamical consequence was addressed. 

\item[(v)] It is well known that $H_{\pi/4}^\kappa$ and $H_+^\kappa$ can have eigenvalues outside the interval $[-2,2]$. For this reason,  we are only interested in \( [-2,2] \).

\item [(vi)] The proof of Theorem \ref{app2} relies on three main ingredients: an asymptotic comparison between the transfer matrix of the free Laplacian and that of the given operator (see Theorem \ref{prop:asymptotic-formula-polynomial} ahead);  an approximation argument for spectral measures in the vein of Carmona, Krutikov, and Remling \cite{Carmona1,Carmona2,Krutikov} (see Theorem \ref{maintheorem0202} ahead); and a sharp Strichartz-type inequality (see Theorem \ref{app1} ahead).

\item [(vii)] It is well known that for each \( -\frac{\pi}{2} < \beta < \frac{\pi}{2} \), the Radon-Nikodym derivative of the spectral measure of the discrete Laplacian \( \triangle_\beta \) on \( \ell^2(\mathbb{Z}^+) \), with boundary condition
\[
u(0) \cos(\beta) + u(1) \sin (\beta) = 0,
\]
 associated with \( \psi = \delta_1 \), is given by~\cite{Carmona2}
\begin{equation}\label{derivadalaplacia}
\rho_{\delta_1}^\beta(x) = \frac{\cos^2(\beta) \sqrt{4 - x^2}}{\pi(2 + \sin(2\beta) x)}, \, \,  x \in (-2,2)
\end{equation}
(for the reader's convenience, we present a proof of this result in Appendix \ref{appendix}). It is then immediate to check that parts~1-2 in Theorem~\ref{app2} hold for \( \kappa = 0 \). Since \( \delta_1 \) is cyclic in this context, Theorem~\ref{app2} shows, in a certain sense, that the spectral and dynamical properties of \( H_{\pi/4}^\kappa \) (\( H_+^\kappa \)) and of \( \triangle_{\pi/4} \) (respectively \( \triangle_0 \)), restricted to the interval \( [-2,2] \), are identical. This reveals a surprising form of rigidity in the absolutely continuous spectrum for small perturbations with sufficiently fast polynomial decay, which, to the best of our knowledge, has not been previously described in the literature. Notably, the rigidity of the absolutely continuous spectrum is closely related to Schr\"odinger's Conjecture~\cite{AvilaS} and its various formulations~\cite{Bruneau}. In this context, our results provide a new perspective on the existing literature in the field.

\item [(viii)] For \( \beta = -\pi/4 \), the result in the statement of Theorem~\ref{app2} coincides with that for \( \beta = \pi/4 \) but with a singularity at $x=2$. For other boundary conditions, the behavior aligns with that of the Dirichlet case. This is directly related to the behavior of the Radon-Nikodym derivative in~\eqref{derivadalaplacia} near \( x = \pm 2 \). Therefore, we restrict our analysis to these boundary conditions. Our arguments seem to extend naturally to the whole-line problem (here, we refer to Theorem~\ref{app2}, part~1-$(iii)$), and the result seems the same as for \( \beta = \pi/4 \), but with singularities at \( x = \pm 2 \). We do not pursue this matter here. 
 
\item [(ix)] We note that $\frac{1}{2}$-H\"older continuity is not sufficient to obtain the ``power-law'' optimal decay rate, \(\log(t)/t\), for the time-averaged quantum return probability. Consider the upper (lower) correlation dimensions of the spectral measure \(\mu_\psi\), which are defined as  
\[
D^+(\mu_{\psi})(D^-(\mu_{\psi})) = \limsup_{\epsilon \downarrow  0} ( \liminf_{\epsilon \downarrow  0} ) \frac{\log \int_{\mathbb{R}} \mu_{\psi}(x-\epsilon,x+\epsilon) \, d\mu_{\psi}(x) }{\log \epsilon}.
\]
It is known that such quantities govern the power-law decay rates of the time-averaged quantum return probability at rate \(t^{-D^\pm}\) \cite{AloCarvCRdOlivro,Barbaroux,Holschneider}. In particular, it suffices for \(D^-(\mu_{\psi})<1\) (which is compatible with $\frac{1}{2}$-H\"older continuity) to preclude a behavior as \(\log(t)/t\). We also note that since $0\leq D^\pm (\mu_{\psi})\leq 1$ \cite{Barbaroux,Holschneider}, one has that the result in the statement of Theorem \ref{app2}  in part~1.\ $(iii)$ is power-law optimal; part~1.\  $(ii)$ is also optimal in the sense that, for a general potential as stated in the theorem, one cannot obtain better than $\frac{1}{2}$-H\"older continuity (simply take $b \equiv 0$).

\item[(x)] We note that numerical simulations in \cite{Mosserit} indicate that, for certain quasi-periodic operators, the time-averaged quantum return probability exhibits the  power-law optimal decay rate \(\log(t)/t\). We do not know a proof of this fact. By taking into account the technique developed in this work, the main difficulty in extending the result to these quasi-periodic models lies in handling the singularities of the spectral measures, of which there can be infinitely many \cite{Li}.
\end{enumerate}

For Schr\"odinger operators \( H = \triangle + V \) on $\ell^2(\mathbb{Z})$, with bounded potentials \( V \), we also present an asymptotic formula (Theorem \ref{transferthm2}) for a family  of positive self-adjoint operators (see \eqref{operators} ahead) associated with the transfer matrix of \( H^\kappa= H+ \kappa b \),  where \(b \) denotes  certain perturbations that decay exponentially fast. This family is intrinsically linked to the analysis of the continuity of the spectral measure through a version of the  Jitomirskaya-Last inequality developed by Avila and Jitomirskaya \cite{AvilaUaH} (see Lemma \ref{lemmaAvila} ahead). We then prove the persistence of \( \frac{1}{2} \)-H\"older continuity of the spectral measure for quasi-periodic operators under exponentially decaying small perturbations. The result also holds for all energies up to the edge of the absolutely continuous spectrum. 

Next, we present a result on exponentially decaying perturbations of analytic quasi-periodic Schr\"odinger operators \( H = H_{\lambda v, \alpha, \theta} \), defined on \(\ell^2(\mathbb{Z})\), by the action
\[
(H u)(n) = u(n+1) + u(n-1) + \lambda v(\theta + n\alpha)u(n), \, \,  n\in\mathbb{Z},
\]
where \( v \in  C^\omega(\mathbb{R}/\mathbb{Z}, \mathbb{R}) \) is the potential (\( C^\omega(\mathbb{R}/\mathbb{Z}, \mathbb{R}) \) stands for real-valued analytic functions on the torus), \( \lambda \in \mathbb{R} \) is the coupling constant, \( \alpha \in \mathbb{R}\setminus \mathbb{Q} \) is the frequency and \( \theta \in \mathbb{R} \) is the phase. Note that the case \( v(x) = 2 \cos(2\pi x) \) corresponds to the almost Mathieu operator.

We recall that a frequency \( \alpha \in \mathbb{T} \) is called Diophantine, denoted by \( \alpha \in DC(\eta, \tau) \), if for  constants \( \eta > 0 \) and \( \tau > 0 \) one has
\[
\left| n\alpha - j \right| > \frac{\eta}{|n|^\tau}, \, \,  \forall n \in \mathbb{Z} \setminus \{0\}, \, \,  \forall j \in \mathbb{Z}.
\]
The set of all Diophantine numbers is defined as
\[
DC := \bigcup_{\eta>0, \tau>0} DC(\eta, \tau).
\]
 
We shall prove the following result.

\begin{theorem}\label{app3}  Let \( v \in C^\omega(\mathbb{R}/\mathbb{Z}, \mathbb{R}) \). Then,  there exists \( \lambda_0 = \lambda_0(v) > 0 \) such that for each \( |\lambda| < \lambda_0 \), each Diophantine number \( \alpha\in DC \), each $a>3$, each $\tau > 4$, each $b \in \ell^{\infty}(\mathbb{Z},\mathbb{R})$ such that
\[
|b(n)| \leq \gamma (a+2|\lambda|||v||_{\infty})^{-\tau|n|}, \, \,  n \in \mathbb{Z}, \, \gamma>0,
\]
and  each $|\kappa| \leq  {1}/{||b||_\infty}$, 

\item[{ (i)}]  $H^\kappa=H+\kappa b$ has purely absolutely continuous spectrum on  $\sigma_{\rm ess}(H^\kappa) =\sigma_{\rm ess}(H)$;
\item[{ (ii)}] $\mu_\psi^{\lambda v, \alpha, \theta,\kappa b}$ is $\displaystyle\frac{1}{2}$-H\"older continuous for each $\psi= f(H^\kappa)\delta_0 + g(H^\kappa)\delta_1$, with $f,g \in C(\sigma_{\rm ess}(H^\kappa))$.
\end{theorem}

\begin{remark}{\rm Here,  we always consider \(H^\kappa\) restricted to the spectral subspace associated with \(\sigma_{\mathrm{ess}}(H^\kappa)\). Accordingly, when we write \(f(H^\kappa)\) or \(g(H^\kappa)\), with \(f,g\in C(\sigma_{\mathrm{ess}}(H^\kappa))\), the functional calculus is understood for this restricted operator.}    
\end{remark}

\begin{remark}
\end{remark}
\begin{enumerate} 

\item [(i)]  Theorem \ref{app3} in case $\kappa=0$ was proved by Avila and Jitomirskaya (Theorem 1.1 in \cite{AvilaUaH}). The result stated in Theorem \ref{app3} $(ii)$ is optimal. Namely, it is known that there exist square-root singularities at the boundaries of gaps, so the modulus of continuity cannot be improved in general \cite{Puig}.

\item[(ii)]  By closely following the proofs in this work, we can obtain results similar to Theorem \ref{app3} for the operators studied in \cite{Liu,Zhao,Zhao2}. Moreover, by combining these arguments with those of Li  {\it et al.}\ in \cite{Li}, it seems possible to obtain a result analogous to Theorem \ref{app3} for local dimensions, in the spirit of \cite{Li}. We do not explore this possibility here.

\item[(iii)] The persistence of the purely absolutely continuous spectrum in the statement of Theorem \ref{app3} is quite interesting. It is common knowledge that purely absolutely continuous spectrum occurs only rarely in one-dimensional Schr\"odinger operators \cite{Remling0}. We note that the exclusion of singular continuous spectrum for continuous quasi-periodic operators is a delicate issue, closely related to works of Eliasson \cite{Eliasson}, Avila \cite{AvilaG}, and Damanik et al. \cite{Damanik4}. Concerning the purely absolutely continuous spectrum in Theorem \ref{app3}, we rely directly on a result of Damanik et al. \cite{Damanik4}. The original contribution of the present work lies in the establishment of H\"older regularity of the spectral measures.

\item [iv)] It is well known that $H^\kappa$ can have eigenvalues outside the essential spectrum $\sigma_{\rm ess}(H^\kappa)$. For this reason,  we are only interested in $\sigma_{\rm ess}(H^\kappa)$.

\end{enumerate}


\subsection{Organization of text and notations}\label{notations}
\

In Section \ref{sec2}, we present asymptotic formulas, related to transfer matrices of bounded Schr\"odinger operators under decaying perturbations. In this section, we also establish a sharp Strichartz-type inequality (Theorem \ref{app1}), and then we prove Theorem \ref{app2}. Section \ref{sec3} is devoted to the proofs of Theorems \ref{transferthm} and \ref{app1}. In Section \ref{sec4}, we present another asymptotic formula (Theorem \ref{transferthm2}), which deals with a family of positive self-adjoint operators (see \eqref{operators} below) associated with the transfer matrices of bounded Schr\"odinger operators under exponentially decaying perturbations. By combining this  analysis with a theory due to Avila and Jitomirskaya \cite{AvilaUaH}, we then prove Theorem \ref{app3}. Subsection~\ref{sec5} is devoted to the proof of Theorem~\ref{transferthm2}. In the appendix, we recall how to obtain the Radon-Nikodym derivative of the spectral measure of the free Laplacian with suited boundary conditions, and how to compute a standard formula related to its transfer matrix. We also recall a proof of a telescoping formula for matrices.

Some words about the notation:  $\{\delta_n := (\delta_{nj})_{j \in \mathbb{Z}} \}$  stands for the canonical basis of $\ell^2(\mathbb{Z})$. Let \( H \) be a self-adjoint operator acting
on \( \ell^2(\mathbb{Z}) \). The resolvent set of \( H \), denoted by \( \rho(H) \), is the set of values \( \lambda \in \mathbb{C} \) for which \( (H - \lambda I)^{-1} \) exists as a bounded operator. The spectrum of \( H \) is  $\sigma(H)= \mathbb{C} \setminus  \rho(H)$.

For each \( \psi \in \ell^2(\mathbb{Z}) \), let \( \mu = \mu_\psi \) denote the spectral measure associated with the pair \( (H,\psi) \). Let $\mu = \mu_{\mathrm{pp}} + \mu_{\mathrm{sc}} + \mu_{\mathrm{ac}}$ be the standard (unique) decomposition~\cite{Oliveira} of $\mu$  into mutually singular (Borel) pure point, singular continuous and absolutely continuous parts,  with corresponding space decomposition:
\[
\ell^2(\mathbb{Z}) = \ell^2_{\mathrm{pp}}(H) \oplus \ell^2_{\mathrm{sc}}(H) \oplus \ell^2_{\mathrm{ac}}(H).
\]

Denote by  \( H_\bullet := H|_{\ell^2_\bullet(H)} \), for \( \bullet \in \{ \mathrm{pp}, \mathrm{sc}, \mathrm{ac} \} \) the restricted operators, and their associated spectral components as
\[
\sigma_{\mathrm{pp}}(H) := \sigma(H_{\mathrm{pp}}), \, \, 
\sigma_{\mathrm{sc}}(H) := \sigma(H_{\mathrm{sc}}), \, \, 
\sigma_{\mathrm{ac}}(H) := \sigma(H_{\mathrm{ac}}),
\]
so that the spectrum of \( H \) can be written as
\[
\sigma(H) = \sigma_{\mathrm{pp}}(H) \cup \sigma_{\mathrm{sc}}(H) \cup \sigma_{\mathrm{ac}}(H).
\]

A finite Borel measure \( \mu \) on \( \mathbb{R} \) is supported on a Borel set \( \Lambda \subset \mathbb{R} \) if \( \mu(\mathbb{R} \setminus \Lambda) = 0 \). Here, the (canonical) spectral measure of a Schr\"odinger operator \( H = \triangle + V\) acting on \( \ell^2(\mathbb{Z}) \) is the unique Borel measure \( \mu \) for which the whole-line \( m \)-function coincides with its Borel transform, that is, 
\[
M(z) = \int \frac{d\mu(E')}{E' - z} = \langle \delta_0, (H - z)^{-1} \delta_0 \rangle + \langle \delta_1, (H - z)^{-1} \delta_1 \rangle.
\]

A Borel set \( \Sigma_{ac} = \Sigma_{ac}(H) \subset \mathbb{R} \) is called a (Lebesgue) essential support of the absolutely continuous spectrum of $H$ if: 1. \( \mu_{\mathrm{ac}}(\mathbb{R} \setminus \Sigma_{\mathrm{ac}} ) = 0 \), and 2. for any other Borel set \( S \subset \mathbb{R} \) such that \( \mu_{\mathrm{ac}}(\mathbb{R} \setminus S) = 0 \), we have $|\Sigma_{\mathrm{ac}}  \setminus S| = 0,$  where \( |\cdot| \) denotes the Lebesgue measure on \( \mathbb{R} \).  We recall that  $\overline{\Sigma_{\mathrm{ac}}}^{\mathrm{ess}} = \sigma_{\mathrm{ac}}(H)$ \cite{LivroDamanik}. We note that the same results hold for Schr\"odinger operators $H_\beta$ on the half-line as in \eqref{halfoperator}, $-\frac{\pi}{2} < \beta <\frac{\pi}{2}$, where  
\[
m^+(z) = m_0^+(z) =  \int \frac{d\mu^+(E')}{E' - z} =  \langle \delta_1, (H - z)^{-1} \delta_1 \rangle \, \,  {\rm and } \, \,  m^+_\beta = R_{-\beta/2\pi} \cdot m^+.
\]
Those are Borel transforms of the half-line spectral measures \(\mu^\beta = \mu_{\delta_1}^\beta\) of the operator \(H_\beta\) \cite{Jitomirskaya2}. Here we make use of the action of \(\mathrm{SL}(2, \mathbb{C})\) on \(\mathbb{C}\),
\[
\begin{pmatrix} a & b \\
c & d \end{pmatrix} \cdot z = \frac{az + b}{cz + d}.
\]

Unless stated otherwise, we use $||A|| = \sqrt{\operatorname{tr}(A A^*)}$ to denote the Frobenius norm of a \(2 \times 2\) matrix \(A\), 
with \(A^*\) standing for the conjugate transpose of \(A\).

In this work, $\hat{f}$ denotes the Fourier transform of a function $f \in {\mathrm L}^1(\mathbb{R})$, and for a finite Borel measure $\mu$ on $\mathbb{R}$ and $s \in \mathbb{R}$, 
\[\hat{\mu}(s) = \int_{\mathbb{R}} e^{-2\pi isx} d\mu(x). \]

Finally, we recall the notion of transfer matrix. Let  $H$ be defined as in~\eqref{operator}. Consider, for every $z \in \mathbb{C}$, the eigenvalue equation  
\begin{equation*}
(Hu)(n) = u(n+1) + u(n-1) + V(n) u(n) = z u(n), \, \,  n \in \mathbb{Z},
\end{equation*}
which can be rewritten as a first-order system by introducing 
\begin{equation*}
A_n = A(n,z) :=
\begin{bmatrix}
z - V(n) & -1 \\
1 & 0
\end{bmatrix}, \, \,   n \in \mathbb{Z},
\end{equation*}
and setting
\begin{equation*}
\Phi(n) =
\begin{bmatrix}
u(n+1) \\
u(n)
\end{bmatrix};
\end{equation*}
then, the eigenvalue equation is equivalent to 
\begin{equation*}
\Phi(n) = T(z, n, 0) \Phi(0), \, \,  n \in\ \mathbb{Z},
\end{equation*}
where \( T(z, n, 0)\) denotes the transfer matrix 
\[
T(z, n, 0) =
\left\{
\begin{array}{ll}
A_n A_{n-1} \cdots A_2 A_1, & n \geq 1, \\
A_0, &  n = 0, \\
A^{-1}_{n+1} A^{-1}_{n} \cdots A^{-1}_1 A^{-1}_0, &  n \leq -1,
\end{array}
\right.
\]
with \( A_0=I \). For the case of the operator on the half-line, this construction is carried out by restricting the definition above to the sites \( n \geq 0 \).


\section{Proof of the Main Theorem}\label{sec2}

\subsection{Asymptotic formula I}
\ 

The following result shows that,  in some sense, an exponentially decaying small perturbation does not affect the asymptotic behavior of the transfer matrix of a Schr\"odinger operator with bounded potential.
\begin{theorem}\label{transferthm} 
Let $H = \triangle + V$ be a Schr\"odinger operator acting on $\ell^2(J)$, where $J \in \{ \mathbb{Z}, \mathbb{Z}^+ \}$, with $V \in \ell^{\infty}(J, \mathbb{R})$. Let $a>3, \tau > 3$, $\gamma > 0$, and let $b \in \ell^{\infty}(J, \mathbb{R})$ be such that
\[
|b(n)| \leq \gamma (a + 2 \| V \|_\infty)^{-\tau |n|}, \, \,  n \in J.
\]
Let $|\kappa| \leq \displaystyle\frac{1}{\| b \|_\infty}$ and let $H^\kappa = H + \kappa b$ be the  perturbed Schr\"odinger operator acting on $\ell^2(J)$. Denote by $T(x,n,0)$ and $T^\kappa(x,n,0)$  the transfer matrices associated with~$H$ and $H^\kappa$, respectively. Then, for each $n \in J$ and each $x \in \sigma(H^\kappa)$,
\begin{eqnarray}\label{AFI1}
T(x, n, 0) - T^\kappa(x, n, 0) = T(x, n, 0) Q(x) + R_n(x),
\end{eqnarray}
\begin{eqnarray}\label{AFI2}
T^\kappa(x, n, 0) - T(x, n, 0) = T^\kappa(x, n, 0) P(x) + K_n(x),
\end{eqnarray}
with $\| R_n \|, \| K_n \| \leq \eta |\kappa| (a + 2 \| V \|_\infty)^{(3 - \tau) |n|}$ and $\| Q \|, \| P \| \leq \eta |\kappa|$ for some $\eta > 0$, which does not depend on $x \in \sigma(H^\kappa)$. Moreover, \( Q\) and \(P \) are  continuous functions on  \(\sigma(H^\kappa) \). 
\end{theorem}

\begin{corollary}\label{maincoror0000} 
Let $H$ and $H^\kappa$ be as in the statement of Theorem \ref{transferthm}. Then, there exist universal constants $\gamma,\eta > 0$ such that for each $n \in J$ and each $x \in \sigma(H^\kappa)$,
\[
r_n + \gamma^{-1} \| T(x, n, 0) \| \leq \| T^\kappa(x, n, 0) \| \leq \gamma \| T(x, n, 0) \| + r_n,
\]
with $r_n \leq \eta |\kappa| (a + 2 \| V \|_\infty)^{(3 - \tau) |n|}$.
\end{corollary}

In the particular case \(V\equiv 0\), the result above can be sharpened by replacing the exponential decay assumption on \(b\) with a sufficiently fast polynomial decay condition.

\begin{theorem}\label{prop:asymptotic-formula-polynomial} Let \(b\in\ell^\infty(\mathbb Z^+,\mathbb R)\) satisfy
\[
M_2(b):=
\sum_{n=1}^{\infty}(n+1)^2|b(n)|<\infty.
\]
Let $H^\kappa=\triangle+\kappa b$ be the perturbed Schr\"odinger operator acting on
\(\ell^2(\mathbb Z^+)\). Denote by
\(T_\triangle(x,n,0)\) and \(T^\kappa(x,n,0)\) the transfer matrices
associated with \(\triangle\) and \(H^\kappa\), respectively. Then there exists \(\varepsilon>0\) such that, for every
\(|\kappa|<\varepsilon\), for each \(n\geq1\) and each \(x\in[-2,2]\),
\begin{equation}\label{eq:poly-comparison-sch-I}
T_\triangle(x,n,0)-T^\kappa(x,n,0)
=
T_\triangle(x,n,0)Q(x)+R_n(x),
\end{equation}
and
\begin{equation}\label{eq:poly-comparison-sch-II}
T^\kappa(x,n,0)-T_\triangle(x,n,0)
=
T^\kappa(x,n,0)P(x)+K_n(x).
\end{equation}
Moreover, there exists a constant \(C>0\), independent of
\(n\), \(x\), and \(\kappa\), such that
\[
\|Q(x)\|+\|P(x)\|
\leq
C|\kappa|M_2(b),
\, \,  x\in[-2,2],
\]
and
\[
\|R_n(x)\|+\|K_n(x)\|
\leq
C|\kappa|(n+1)
\sum_{j=n+1}^{\infty}(j+1)^2|b(j)|.
\]
Moreover, \(Q\) and \(P\) are continuous functions on \([-2,2]\).
\end{theorem}
\begin{remark}{\rm To prove Theorem \ref{app2}, we combine Theorem \ref{prop:asymptotic-formula-polynomial} with an approximation argument for spectral measures developed by Carmona, Krutikov, and Remling \cite{Carmona1,Carmona2,Krutikov}. One point that requires a slightly more refined analysis when applying Carmona’s Theorem (see Theorem \ref{CarmonaTHM} below) is the fact that the matrix \( Q \) in the statement of Theorem \ref{prop:asymptotic-formula-polynomial} appears on the right-hand of the transfer matrix. However, upon careful reflection, one sees that this cannot be otherwise, and indeed, this structure is essential: it prevents us from proving that the \( \frac{1}{2} \)-H\"older continuity behavior stated in part 1. $(ii)$ of Theorem \ref{app2} is optimal for every potential as described in the statement of the theorem. In fact, it is straightforward to construct examples involving finite-rank perturbations for which the spectral measures are Lipschitz continuous up to the edge of the absolutely continuous spectrum in the setting of part 1. of Theorem \ref{app2}.}
\end{remark}


\subsection{Sharp Strichartz-type inequality}
\ 

Strichartz's Theorem \cite{strichartz1990} (see Theorem \ref{Strichartztheorem} (i) below)  establishes (power-law) convergence rates for the time-average of the squared absolute value of the Fourier transform of $\alpha$-H\"older continuous measures. We now proceed to state it.

\begin{theorem}[Theorems 2.5 and 3.1  in \cite{Last}]\label{Strichartztheorem} Let $\mu$ be a finite Borel measure on $\mathbb{R}$ and let $\alpha \in [0,1]$.

\begin{enumerate}

\item[(i)] If $\mu$ is $\alpha$-H\"older continuous, then there exists a constant $C_\mu> 0$, depending only on $\mu$, such that for every $f \in {\mathrm L}^2(\mathbb{R}, d\mu )$ and every $t>0$, 
\[\frac{1}{t} \int_0^t \bigg|\int_{\mathbb{R}} e^{-2\pi isx} f(x)\,  d\mu(x) \bigg|^2  ds \leq C_\mu \|f\|_{{\mathrm L}^2(\mathbb{R}, d\mu )} t^{-\alpha}. \] 

\item[(ii)] If there exists $C_{\mu}>0$ such that for every $t>0$, 
\[\frac{1}{t} \int_0^t \bigg|\int_{\mathbb{R}} e^{-2\pi isx}\,  d\mu(x) \bigg|^2  ds \leq C_\mu \,  t^{-\alpha},\] 
then $\mu$ is $\frac{\alpha}{2}$-H\"older continuous.

\end{enumerate}
\end{theorem} 

\begin{remark}{\rm It is worth noting that Theorem \ref{Strichartztheorem}-$(i)$ is, indeed, a particular case of Strichartz's Theorem~\cite{strichartz1990}.}
\end{remark}

Let $\nu$ be a positive Borel measure on $\mathbb{R}$ that is absolutely continuous with respect to the Lebesgue measure, with Radon-Nikodym derivative given by $f$. We say that $\nu$ has a singularity of type $\varrho \in [0,1)$ if there exist a constant $\gamma > 0$, a sequence of disjoint intervals $[a_j, a_{j+1})$ and a sequence of positive numbers, $\varrho_j\le\varrho$, $j = 1, \dots, n$, such that for each $j$, there exists at most one point $b_j \in [a_j, a_{j+1})$ satisfying
\[
f(x) = \sum_{j=1}^n f_j(x), \,  x \in \mathbb{R},
\]
where, for each $j=1,\ldots,n$, 
\[f_j(x) \leq \gamma \cdot \frac{\chi_{[a_j, a_{j+1})}(x)}{|x - b_j|^{\varrho_j}}, \,  x \in \mathbb{R}.\]

We say that a positive Borel measure $\mu$ has a weak-singularity of type $\varrho \in [0,1]$ if there exists a constant $\gamma' > 0$ such that, for every continuous non-negative real-valued function $f$,
\begin{equation*}
\int_{\mathbb{R}} f(x) \, d\mu(x) \leq \gamma' \int_{\mathbb{R}} f(x) \, d\nu(x),
\end{equation*}
where \( \nu \) denotes the measure previously defined.

The next result, which is a refinement of Theorem \ref{Strichartztheorem}, is used in the proof of Theorem \ref{app2}. Among other things, it establishes that the asymptotic behavior of the squared absolute value of the Fourier transform of certain measures may depend continuously on their singularities. We note that this refines some previous results,  for much more specific measures, presented in~\cite{Aloisio}.

\begin{theorem}\label{app1} Let $\mu$ be as before. Suppose that $\mu$ has a weak-singularity of type  $\varrho \in [0,1)$. Then, there exists a constant $C_\mu> 0$ such that:
\begin{enumerate}

\item[(i)] if $0\leq \varrho < \frac{1}{2}$, then 
\begin{eqnarray*}
\frac{1}{t} \int_0^t \bigg|\int_{\mathbb{R}} e^{-2\pi isx}\,  d\mu(x) \bigg|^2  ds  \leq C_\mu \, \frac{1}{t}, \;\;\; \forall t>0;
\end{eqnarray*}

\item[(ii)] if $\varrho = \frac{1}{2}$, then 
\begin{eqnarray*}
\frac{1}{t} \int_0^t \bigg|\int_{\mathbb{R}} e^{-2\pi isx}\,  d\mu(x) \bigg|^2  ds   \leq C_\mu   \,  \frac{\log(t)}{t}, \;\;\; \forall t>2;
\end{eqnarray*}

\item[(iii)] if $\frac{1}{2}< \varrho <1$, then 
\begin{eqnarray*}
\frac{1}{t} \int_0^t \bigg|\int_{\mathbb{R}} e^{-2\pi isx}\,  d\mu(x) \bigg|^2  ds    \leq C_\mu   \,  \frac{1}{t^{2(1-\varrho)}}, \;\;\; \forall t>0. 
\end{eqnarray*}
\end{enumerate}
\end{theorem}

\ 

 Example below shows that one cannot directly apply Theorem~\ref{Strichartztheorem} to prove Theorem~\ref{app1}. 

\begin{example}[Example 3.1 in \cite{Last}]\label{mainexample}{\rm Let $\frac{1}{2} < \beta < 1$  and let  $\nu_\beta$ be the measure defined by $\int_0^1 f(x) \, d\nu_\beta(x) = \int_0^1 f(x)x^{-\beta} \, dx, \;\;\; f \in C([0,1]).$ Then, there exists  $C_{\nu_\beta} >0$ such that
\begin{equation}\label{eqAvFTransMeas}
\frac{1}{t} \int_0^t | \hat{\nu}(s) |^2  ds < C_{\nu_\beta} \, t^{-2(1-\beta)}, \;\;\; \forall t>0,
\end{equation}
and $\nu_\beta$ is at most $(1-\beta)$-H\"older continuous.}
\end{example}

We note also that this example demonstrates that a measure which is at most $\alpha$-H\"older continuous exhibits a time-average decay of $t^{-2\alpha}$ as in~\eqref{eqAvFTransMeas}. Although very simple, it provides valuable insight into the singularities of continuous measures, particularly those  purely absolutely continuous spectral measures.

 \begin{remark} \label{remNBastaTrunc} {\rm  Let \(\delta \in (0,1/2) \) and  let \(\epsilon \in(0,1) \); then, although \( \nu_{1-\delta} \) (Example \ref{mainexample}) is at most  $\delta$-H\"older continuous, the measure 
\[
d\nu_{(1-\delta)}^\epsilon(x) = \chi_{(\epsilon,1]}(x)\, d\nu_{1-\delta}(x)
\]
is Lipschitz continuous.  This illustrates that, to effectively study the singularities of continuous spectral measures, it is not sufficient to rely solely on a strong convergence theorem for truncated versions of the measure; one must also establish a uniform comparison with an object whose behavior is well understood on the entire support (for example, another measure). Overcoming this challenge was one of the key difficulties we have faced.}
 \end{remark}


\subsection{Proof of Theorem \ref{app2}}
\ 

Next, we present a proof of Theorem \ref{app2}. However, some preparation is required.

\begin{theorem}[Theorem 111.3.6 in \cite{Carmona2}]\label{CarmonaTHM}
Let $H_\beta$ be as in \eqref{halfoperator}, and let $u_\beta = (\cos(\beta), -\sin(\beta))$. Then, for every continuous function with compact support $f$, one has
\[
\lim_{n \to \infty} \frac{1}{\pi} \int f(x) \|T(x, n, 0)u_\beta\|^{-2} \, dx = \int f(x) \, d\mu_{\delta_1}^\beta(x),
\]
where $\mu_{\delta_1}^\beta$ is the spectral measure associated with $H_\beta$ and  the cyclic vector $\delta_1$.
\end{theorem}

\begin{remark}
\end{remark}
\begin{itemize}
    \item[(i)] Carmona set the boundary condition as $u(0) \cos(\beta) - u(1) \sin(\beta) = 0$, whereas we use $u(0) \cos(\beta) + u(1) \sin(\beta) = 0$. Therefore, naturally, $u_\beta = (\cos(\beta), \sin(\beta))$ in \cite{Carmona1,Carmona2}.
    
    \item[(ii)] We also note that Krutikov and Remling have proved this result for Dirichlet boundary condition (see Corollary 2.2 in \cite{Krutikov}).
\end{itemize}
 
The next lemma follows from straightforward calculations, but for the reader's convenience, we provide its proof in Appendix~\ref{appendix}.

\begin{lemma}\label{teclemmamainthm} Let
\[
A(x) = \begin{pmatrix} x & -1 \\ 1 & 0 \end{pmatrix}, \, \,  x \in (-2,2),
\, \,  \text{and} \, \,  \theta = \arccos\left( \frac{x}{2} \right) \in (0,\pi).
\]
Then, one has 
$A = P D P^{-1}$ and, for each $\beta \in \mathbb{R}$ and every $n \geq 1$,
\[\left\| D^n P^{-1} (\cos(\beta),-\sin(\beta))^{\mathrm T} \right\|^2 = \displaystyle\frac{1 + \cos(\theta) \sin(2\beta)}{2\sin^2(\theta)},\]
where 
\[
P = \begin{pmatrix}
1 & 1 \\
e^{-i \theta} & e^{i \theta}
\end{pmatrix} \, \,  {\rm and} \, \,  D = \begin{pmatrix}
e^{i \theta} & 0 \\
0 & e^{-i \theta}
\end{pmatrix}.
\]
\end{lemma}

In what follows, \(T_{\triangle}\) refers to the transfer matrix of the free Laplacian. For each \(\Omega \subset \mathbb{R}\), we denote by \(C_c(\Omega)\) the space of continuous functions \(f : \mathbb{R} \to \mathbb{C}\) whose support is compact and contained in \(\Omega\).

\begin{lemma}\label{mainmainlema} Let \( \beta \in C([-2,2]) \) with \({\rm Ran}(\beta) \subset (-\pi/2, \pi/2)\) and let, for every $x \in (-2,2)$, \( u_{\beta(x)} = (\cos(\beta(x)), -\sin(\beta(x)))\). Set, for every $x \in (-2,2)$,
\[
\rho_{\beta(x)}(x) = \frac{\cos^2(\beta(x)) \cdot \sqrt{4 - x^2}}{\pi(2 + \sin(2\beta(x))x)}.
\]
Then, for every $\epsilon>0$ and  every \( f_\epsilon \in C_c([-2+\epsilon,2-\epsilon]) \),
\[
\lim_{n \to \infty} \frac{1}{\pi} \int f_\epsilon(x) \, \|T_{\triangle}(x,n,0) u_{\beta(x)}\|^{-2} \, dx = \int f_\epsilon(x) \cdot \rho_{\beta(x)}(x) \, dx.
\]
\end{lemma}

\begin{proof} First, let \(P\) and \(D\) be the matrices from Lemma \ref{teclemmamainthm}. Note that, for each \( x \in (-2,2) \) and each unit vector \( v(\alpha) = (\cos(\alpha), -\sin(\alpha)) \), \( \alpha \in \mathbb{R}\),
\begin{eqnarray}\label{eqemin01}\nonumber
\|T_{\triangle}(x,n,0) v(\alpha)\| &\leq& \|T_{\triangle}(x,n,0)\| = \|PD^nP^{-1}\|\\ &\leq& \|P\| \|D^n\| \|P^{-1}\| \leq 4 \|P^{-1}\|  = \frac{8}{\sqrt{4-x^2}},    
\end{eqnarray}
and 
\begin{eqnarray}\label{eqemin02} \nonumber
\|T_{\triangle}(x,n,0) v(\alpha)\|^2 &\geq& ||P^{-1}||^{-2} \|D^nP^{-1}v(\alpha)\|^2\\ \nonumber
&=&  \frac{1 + \cos(\theta) \sin(2\alpha)}{2}  \geq \frac{1 - |\cos(\theta) \sin(2\alpha)|}{2} \\ 
&\geq& \frac{1 - |\cos(\theta)|}{2} = \frac{1}{2} - \frac{|x|}{4}.
\end{eqnarray}
Thus, the map \( x \mapsto \|T_{\triangle}(x,n,0) u_{\beta(x)}\|^{-2} \) is strictly positive and uniformly bounded in \(n\) on compact subsets of \( (-2,2) \).

Let \( \{ \beta^{(N)} \} \) be a sequence of step functions that converges pointwise to \( \beta \) as \( N \to \infty \), that is,
\[\beta^{(N)} = \sum_{j} a_j \, \chi_{I_j},\]
and set
\[
u^{(N)}(x) = (\cos(\beta^{(N)}(x)), -\sin(\beta^{(N)}(x))), \, \,  x \in (-2,2).
\]
 
Now, for each \( N \), the function \( u^{(N)} \) takes finitely many constant values \( u_j = (\cos(\beta_j), -\sin(\beta_j)) \) on the interval  \( I_j \) and so, for every $\epsilon>0$ and every $f \in C_c([-2+\epsilon,2-\epsilon])$,
\begin{equation}\label{eqeqeqeq0101010}
\lim_{n \to \infty} \frac{1}{\pi} \int f_\epsilon(x) \, \|T_{\triangle}(x,n,0) u^{(N)}(x)\|^{-2} \, dx = \int f_\epsilon(x) \cdot \rho^{(N)}(x) \, dx,    
\end{equation}
where \( \rho^{(N)}(x) = \rho_j(x) = \displaystyle\frac{\cos^2(\beta_j) \cdot \sqrt{4 - x^2}}{\pi(2 + \sin(2\beta_j)x)} \) for \( x \in I_j \). Namely, for each \( I_j \), let \( \{\varphi_\delta\} \subset C_c([-2,2]) \) be a sequence of continuous functions converging pointwise to \( \chi_{I_j} \) as \( \delta \to 0^+ \) such that ${\rm supp } \varphi_\delta \subset I_j$ and $0 \leq \varphi_\delta \leq 1$. We note that, for every \( \epsilon > 0 \) and every \( f \in C_c([-2+\epsilon, 2-\epsilon]) \), there exists $C_\epsilon>0$ such that 
\begin{eqnarray*}
&\,& \biggr| \int_{I_j} f_\epsilon(x) \, \|T_{\triangle}(x,n,0) u^{(N)}(x)\|^{-2} \, dx\\ 
&-& \int f_\epsilon(x) \varphi_\delta(x) \, \|T_{\triangle}(x,n,0) u^{(N)}(x)\|^{-2} \, dx \biggr|\\ 
&\leq& C_\epsilon \int |\varphi_\delta(x) - \chi_{I_j}(x)| \, dx \to 0
\end{eqnarray*}
as \( \delta \to 0^+ \), uniformly in \( n \) and $N$, since, by \eqref{eqemin01} and \eqref{eqemin02}, the map \( x \mapsto \|T_{\triangle}(x,n,0) u^{(N)}(x)\|^{-2} \) is uniformly bounded in \( n \)  and $N$ on compact subsets of \( (-2,2) \). We also note that 
\begin{eqnarray*}
\lim_{n \to \infty} &\frac{1}{\pi}& \int \varphi_\delta(x) f_\epsilon(x) \, \|T_{\triangle}(x,n,0) u^{(N)}(x)\|^{-2} \, dx\\ &=& \lim_{n \to \infty} \Bigg\{ \int_{I_j} f_\epsilon(x) \, \|T_{\triangle}(x,n,0) u^{(N)}(x)\|^{-2} \, dx \\
 &+& \biggr( \int f_\epsilon(x) \varphi_\delta(x) \, \|T_{\triangle}(x,n,0) u^{(N)}(x)\|^{-2} \, dx\\ 
&-& \int_{I_j} f_\epsilon(x) \, \|T_{\triangle}(x,n,0) u^{(N)}(x)\|^{-2} \, dx \biggr) \Bigg\},
\end{eqnarray*}
so
\begin{eqnarray*}
\lim_{n \to \infty} &\frac{1}{\pi}& \int \varphi_\delta(x) f_\epsilon(x) \, \|T_{\triangle}(x,n,0) u^{(N)}(x)\|^{-2} \, dx\\ 
&\leq& \lim_{n \to \infty}  \int_{I_j} f_\epsilon(x) \, \|T_{\triangle}(x,n,0) u^{(N)}(x)\|^{-2} \, dx\\ 
&+& C_\epsilon \int |\varphi_\delta(x) - \chi_{I_j}(x)| \, dx, 
\end{eqnarray*}
and
\begin{eqnarray}\label{bisbis}\nonumber
\lim_{n \to \infty} &\frac{1}{\pi}& \int \varphi_\delta(x) f_\epsilon(x) \, \|T_{\triangle}(x,n,0) u^{(N)}(x)\|^{-2} \, dx\\ \nonumber
&\geq& \lim_{n \to \infty}  \int_{I_j} f_\epsilon(x) \, \|T_{\triangle}(x,n,0) u^{(N)}(x)\|^{-2} \, dx\\ 
&-& C_\epsilon \int |\varphi_\delta(x) - \chi_{I_j}(x)| \, dx. 
\end{eqnarray}
Then, it follows from Theorem~\ref{CarmonaTHM}, the identity~\eqref{derivadalaplacia}, and the fact that \({\rm Ran}(\beta) \subset (-\pi/2, \pi/2)\),
\begin{eqnarray*}
\lim_{n \to \infty} \frac{1}{\pi} \int \varphi_\delta(x) f_\epsilon(x) \, 
\|T_{\triangle}(x,n,0) u^{(N)}(x)\|^{-2} \, dx 
= \int \varphi_\delta(x) f_\epsilon(x) \cdot \rho_j(x) \, dx;
\end{eqnarray*}
by letting \( \delta \to 0^+\), it follows from \eqref{bisbis} that 
\begin{eqnarray*}
\lim_{n \to \infty} \frac{1}{\pi} \int_{I_j} f_\epsilon(x) \, 
\|T_{\triangle}(x,n,0) u^{(N)}(x)\|^{-2} \, dx 
= \int_{I_j} f_\epsilon(x)  \cdot \rho_j(x) \, dx,
\end{eqnarray*}
where the last equality follows from dominated convergence, since for every \( x \in (-2,2) \) and every~$j$, 
\[
\rho_j(x)  = \frac{\cos^2(\beta_j) \cdot \sqrt{4 - x^2}}{\pi(2 + \sin(2\beta_j)x)} \leq \frac{\sqrt{4 - x^2}}{\pi(2 + x)} + \frac{\sqrt{4 - x^2}}{\pi(2 - x)} = \frac{\sqrt{|2 - x|}}{\pi \sqrt{|2 + x|}} + \frac{ \sqrt{|2 + x|}}{\pi \sqrt{|2 - x|}}.
\]

Now, for every $\epsilon>0$ and for every $x \in [-2+\epsilon,2-\epsilon]$, there exist $C_\eps',C_\eps^{"} >0$ such that 
\begin{eqnarray}\label{mainlemaeqeqeq010101} \nonumber
&\,&\left| \|T_{\triangle}(x,n,0) u_{\beta(x)}\|^{-2} - \|T_{\triangle}(x,n,0) u^{(N)}(x)\|^{-2} \right|\\ \nonumber &\leq&  C_\eps'    \left| \|T_{\triangle}(x,n,0) u_{\beta(x)}\|- \|T_{\triangle}(x,n,0) u^{(N)}(x)\|\right| \\ \nonumber &\leq&  C_\eps'     \|T_{\triangle}(x,n,0) \| \|u_{\beta(x)}-u^{(N)}(x)\| \\ &\leq&  C_\eps^{"}    \|u_{\beta(x)}-u^{(N)}(x)\| \to 0
\end{eqnarray}
as $N \to \infty$, uniformly in $n$. It is worth noting that above we have used the following facts: 1) the map \( x \mapsto \|T_{\triangle}(x,n,0) u_{\beta(x)}\| \) is strictly positive and uniformly bounded in \(n\) on compact subsets of \( (-2,2) \), and  the function \( f(x) = \displaystyle\frac{1}{x^2} \) is Lipschitz continuous on any compact subset contained in \( (0, \infty) \); 2)  \( \|T_{\triangle}(x,n,0)\| \) is  uniformly bounded in \(n\) on compact subsets of \( (-2,2) \). 

Finally, for every $\epsilon>0$ and every $f_\epsilon \in C_c([-2+\epsilon,2-\epsilon])$, by following the same argument as in \eqref{bisbis}, and by using  \eqref{eqeqeqeq0101010} and \eqref{mainlemaeqeqeq010101}, we have
\begin{eqnarray*}
 \lim_{n \to \infty} &\frac{1}{\pi}& \int f_\epsilon(x) \, \|T_{\triangle}(x,n,0) u_{\beta(x)}\|^{-2} \, dx\\  &=& \lim_{N \to \infty} \lim_{n \to \infty} \frac{1}{\pi} \int f_\epsilon(x) \, \|T_{\triangle}(x,n,0) u^{(N)}(x)\|^{-2} \, dx\\ &=& \lim_{N \to \infty} \int f_\epsilon(x) \cdot \rho^{(N)}(x) \, dx\\ &=& \int f_\epsilon(x)  \cdot \rho_{\beta(x)}(x) \, dx,    
\end{eqnarray*}
where the last equality follows again by dominated convergence, since for every \( x \in (-2,2) \) and every~$N$, 
\[
\rho_N(x)  \leq  \frac{\sqrt{|2 - x|}}{\pi \sqrt{|2 + x|}} + \frac{ \sqrt{|2 + x|}}{\pi \sqrt{|2 - x|}}.
\]
\end{proof}

We first recall the notation used above. We denote by
$\mu^+_\kappa=\mu^+_{\delta_1}$ the spectral measure associated with
$H^\kappa_+$ and the cyclic vector $\delta_1$, and by
$\mu^{\pi/4}_\kappa=\mu^{\pi/4}_{\delta_1}$ the spectral measure associated with
$H^\kappa_{\pi/4}$ and $\delta_1$.

\begin{lemma}\label{lema-massa0}
Assume the hypotheses of Theorem~\ref{app2}. Then, for $|\kappa|$
sufficiently small,
\[
\mu^+_\kappa(\{-2\})=\mu^+_\kappa(\{2\})=0
\]
and
\[
\mu^{\pi/4}_\kappa(\{-2\})
=
\mu^{\pi/4}_\kappa(\{2\})
=
0.
\]
\end{lemma}

\begin{proof} Let $u_\beta = (\cos(\beta), -\sin(\beta))$, $\beta = \pi/4$ or $\beta = 0$, and let $E\in\{-2,2\}$. By Theorem~\ref{prop:asymptotic-formula-polynomial},
\[
T^\kappa(E,n,0)u_\beta
=
T_\triangle(E,n,0)(I-Q(E))u_\beta
-
R_n(E)u_\beta,
\]
where
\[
\|R_n(E)\|
\leq
C|\kappa| (n+1)^{4-\tau}.
\]
For $|\kappa|$ sufficiently small, $\|Q(E)\|<1$, and therefore
\[
v(E):=(I-Q(E))u_\beta\neq 0.
\]

We claim that
\[
\sum_{n=1}^\infty
\|T_\triangle(E,n,0)v(E)\|^2
=
\infty.
\]
Indeed, if $E=2$, every solution of
\[
u(n+1)+u(n-1)=2u(n)
\]
is of the form
\[
u(n)=c_1+c_2n.
\]
If $E=-2$, every solution of
\[
u(n+1)+u(n-1)=-2u(n)
\]
is of the form
\[
u(n)=(-1)^n(c_1+c_2n).
\]
Since $v(E)\neq0$, the corresponding solution is nontrivial and,
therefore, does not belong to $\ell^2(\mathbb Z_+)$. Hence,
\[
\sum_{n=1}^\infty
\|T_\triangle(E,n,0)v(E)\|^2
=
\infty.
\]

On the other hand, since $\tau>5$,
\[
\sum_{n=1}^\infty
\|R_n(E)u_\beta\|^2
<
\infty.
\]
Moreover,
\[
\|T_\triangle(E,n,0)v(E)\|
\leq
\|T^\kappa(E,n,0)u_\beta\|
+
\|R_n(E)u_\beta\|.
\]
Therefore,
\[
\|T_\triangle(E,n,0)v(E)\|^2
\leq
2\|T^\kappa(E,n,0)u_\beta\|^2
+
2\|R_n(E)u_\beta\|^2.
\]
Since the left-hand side is not summable and the second term on the
right-hand side is summable, it follows that
\[
\sum_{n=1}^\infty
\|T^\kappa(E,n,0)u_\beta\|^2
=
\infty.
\]
Since
\[
T^\kappa(E,n,0)u_\beta
=
\begin{pmatrix}
u(n+1)\\
u(n)
\end{pmatrix},
\]
where \(u\) is the solution of \(H^\kappa_\beta u=Eu\) satisfying the boundary condition determined by \(\beta\), the divergence
\[
\sum_{n=1}^\infty
\|T^\kappa(E,n,0)u_\beta\|^2
=
\infty
\]
shows that \(u\notin \ell^2(\mathbb Z_+)\). Therefore, \(E\) is not an eigenvalue of \(H^\kappa_\beta\). Hence,
\[
\mu^\beta_{\delta_1}(\{E\})=0.
\]
\end{proof}

The next lemma follows from a direct computation of the free transfer matrix, which shows linear growth in \(n-m\), uniformly for \(x\in[-2,2]\). Similar calculations can be found in the Appendix, and therefore we omit the proof.

\begin{lemma}\label{lem:free-linear-growth}
There exists a constant $C_0>0$ such that, for every
$0\leq m\leq n$ and every $x\in[-2,2]$,
\[
\left\|
T_\triangle(x,n,0)T_\triangle(x,m,0)^{-1}
\right\|
\leq C_0(1+n-m).
\]
In particular,
\[
\|T_\triangle(x,n,0)\|
+
\|T_\triangle(x,n,0)^{-1}\|
\leq C_0(n+1),
\, \, 
n\geq 0,\, \,  x\in[-2,2].
\]
\end{lemma}

As we shall see below, Theorem~\ref{app2} follows from the next result, where we establish a direct comparison between the spectral measure of the perturbed operator and that of the free Laplacian.

\begin{theorem}\label{maintheorem0202} Under the same assumptions as in the statement of Theorem~\ref{app2}, there exists a constant \( \gamma > 0 \) such that, for every non-negative function \( f \in C([-2,2]) \),
\[ \int_{-2}^2 f(x) \, d\mu(x) \leq \gamma \int_{-2}^2 f(x) \, d\mu^{\triangle}(x),\]
where (i) $\mu^\triangle = \mu_{\delta_1}^{\triangle_{+}}$ and $\mu = \mu^{+}_\kappa = \mu_{\delta_1}^+$ or (ii) $\mu^\triangle = \mu_{\delta_1}^{\triangle_{\pi/4}} + \mu_{\delta_1}^{\triangle_{-\pi/4}}$ and $\mu = \mu^{\pi/4}_\kappa = \mu_{\delta_1}^{{\pi/4}}$.
\end{theorem}

\begin{proof} We fix $|\kappa| < \dfrac{1}{100 \eta}$, where $\eta  = C M_2(b)> 0$  is given in the statement of Theorem \ref{prop:asymptotic-formula-polynomial} and depends only on $b$.
Let $x \in [-2,2]$ and let  
\[
A(x) = 
\begin{pmatrix}
x & -1 \\
1 & 0
\end{pmatrix},  \, \,  Q(x) = 
\begin{pmatrix}
q_{11}(x) & q_{12}(x) \\
q_{21}(x) & q_{22}(x)
\end{pmatrix},
\]
where the matrix $Q(x)$ is given by Theorem \ref{prop:asymptotic-formula-polynomial}. 

First, we prove $(i)$. Set $u_0=(1,0)$. Then, for every $x \in [-2,2]$ and every $n \geq 1$,  
\begin{eqnarray*}
T_\triangle(x, n, 0)\, Q(x) u_0 &=&   T_\triangle(x, n, 0)
\begin{pmatrix}
q_{11}(x) \\
q_{21}(x)
\end{pmatrix}\\
&=& q_{11}(x)\, A^n(x) 
\begin{pmatrix}
1 \\
0
\end{pmatrix}
+ q_{21}(x)\, A^n(x) 
\begin{pmatrix}
0 \\
1
\end{pmatrix} \\
&=& q_{11}(x)\, A^n(x) 
\begin{pmatrix}
1 \\
0
\end{pmatrix}
- q_{21}(x)\, A^{-1}(x) A^n(x) 
\begin{pmatrix}
1 \\
0
\end{pmatrix}.
\end{eqnarray*}
Using that, for each $x \in [-2,2]$, $\|Q(x)\| \leq \eta |\kappa| < \displaystyle\frac{1}{100}$, we have that, for every $x \in [-2,2]$ and every $n \geq 1$,  
\[
\|T_\triangle(x, n, 0)Q(x)u_0\| \leq  \frac{4}{100} \|T_\triangle(x, n, 0)u_0\|.
\]
Hence, by Theorem \ref{prop:asymptotic-formula-polynomial},  
\begin{eqnarray*}
\|T_\triangle(x, n, 0)u_0\| &\leq& \|T^\kappa(x, n, 0)u_0 - T_\triangle(x, n, 0)u_0\| + \|T^\kappa(x, n, 0)u_0\|\\
&\leq&  \|T_\triangle(x, n, 0)Q(x)u_0\|  + \|T^\kappa(x, n, 0)u_0\| + r_n\\
&\leq& \frac{4}{100} \|T_\triangle(x, n, 0)u_0\|   + \|T^\kappa(x, n, 0)u_0\|  + r_n,   
\end{eqnarray*}
where $r_n \lesssim |\kappa|(n+1)^{(4-\tau)}$. Thus, for every $x \in [-2,2]$ and every $n \geq 1$, 
\begin{eqnarray*}
\|T_\triangle(x, n, 0)u_0\| &\leq& 2\|T^\kappa(x, n, 0)u_0\| + 2r_n.    
\end{eqnarray*}
Note that, by Lemma \ref{lem:free-linear-growth} and due to the fact that transfer matrices are unimodular,  
\begin{eqnarray*}
\lim_{n \to \infty} \frac{r_n}{\|T_\triangle(x, n, 0)u_0\|}  &\leq& \lim_{n \to \infty} r_n \|T_\triangle^{-1}(x, n, 0)\|\\ 
&=& \lim_{n \to \infty} r_n\|T_\triangle(x, n, 0)\| = 0    
\end{eqnarray*}
uniformly in $x \in [-2,2]$, since we assumed that $\tau > 5$ and, by Lemma \ref{lem:free-linear-growth}, $\|T_\triangle(x, n, 0)\| \lesssim n$.  Thus, there exists $n_0>1$ such that, for every $x \in [-2,2]$ and for all $n>n_0$,
\begin{eqnarray}\label{eq0909090} \nonumber
\|T^\kappa(x, n, 0)u_0\|^{-2} &\leq&  4\|T_\triangle(x, n, 0)u_0\|^{-2}\left( 1 - \frac{2r_n}{\|T_\triangle(x, n, 0)u_0\|}   \right)^{-2}\\ &\leq& 8\|T_\triangle(x, n, 0)u_0\|^{-2}.
\end{eqnarray}

Finally, let \( f \in C([-2,2]) \) be non-negative function and consider a continuous extension of \(f\) to  $\mathbb{R}$, which, for simplicity, we will also denote by \(f\). Let $\{\varphi_\epsilon\} \subset C_c([-2,2])$ be a sequence of positive continuous functions that converges pointwise to $\chi_{(-2,2)}$ as $\epsilon \to 0^+$, with $\varphi_\epsilon \equiv 1$ on $[-2+\epsilon, 2-\epsilon]$, $\varphi_\epsilon \equiv 0$ on $[-2,-2+\frac{\epsilon}{2}] \cup [2-\frac{\epsilon}{2},2]$, and $0\leq \varphi_\epsilon \leq 1$. Then, by Theorem \ref{CarmonaTHM} and \eqref{eq0909090}, there exists $\gamma>0$ such that,    
\begin{eqnarray*}\label{eqcarmona} \nonumber
\int f(x) \varphi_\epsilon(x) \, d\mu(x) &=& \lim_{n \to \infty} \frac{1}{\pi} \int  f(x) \varphi_\epsilon(x)  \left\| T^\kappa(x, n, 0)u_0 \right\|^{-2} \, dx\\  \nonumber
&\leq& \gamma  \lim_{n \to \infty} \frac{1}{\pi} \int f(x) \varphi_\epsilon(x)  \left\| T_\triangle(x, n, 0)u_0 \right\|^{-2} \, dx \\ 
&=& \gamma \int f(x) \varphi_\epsilon(x) \, d\mu^{\triangle}(x);
\end{eqnarray*}
by dominated convergence and Lemma \ref{lema-massa0}, 
\begin{eqnarray*}
\int_{-2}^2 f(x) \, d\mu(x) \leq \gamma \int_{-2}^2 f(x)  \, d\mu^{\triangle}(x).
\end{eqnarray*}

Now, we present a proof of $(ii)$. First, we write, for every $x \in [-2,2]$ and every $n \geq 1$, 
\begin{eqnarray}\label{eq1010101010lalalala} \nonumber
&\,&T_\triangle(x, n, 0)\, (I-Q(x)) u_{\pi/4}\\ \nonumber  &=& ||v(x)|| T_\triangle(x, n, 0)
\begin{pmatrix}
(\sqrt{2}/2 - \varepsilon_1(x))/||v(x)|| \\
(-\sqrt{2}/2 - \varepsilon_2(x))/||v(x)||
\end{pmatrix}\\ \nonumber &=&  ||v(x)|| T_\triangle(x, n, 0)
\begin{pmatrix}
\cos(\beta(x)) \\
-\sin(\beta(x))
\end{pmatrix}\\ &=& ||v(x)|| T_\triangle(x, n, 0) u_{\beta(x)},
\end{eqnarray}
where $\beta \in C([-2,2])$ and \({\rm Ran}(\beta) \subset (-\pi/2, \pi/2)\), since $Q \in C([-2,2])$ and, for each $j=1,2$ and for each $x \in [-2,2]$, $|\varepsilon_j(x)|  \leq  2||Q(x)||<\displaystyle\frac{1}{50}$; here,
\[
v(x) = \begin{pmatrix}
\sqrt{2}/2 - \varepsilon_1(x) \\
-\sqrt{2}/2 - \varepsilon_2(x)
\end{pmatrix}
\]
and so, for every $x \in [-2,2]$, $||v(x)|| \geq \displaystyle\frac{1}{2}$. 

Then, by Theorem \ref{prop:asymptotic-formula-polynomial}  and \eqref{eq1010101010lalalala}, for each $x \in [-2,2]$ and each $n \geq 1$, 
\begin{eqnarray*}
||T^\kappa (x, n, 0)u_{\pi/4}|| &\geq& ||T_\triangle(x, n, 0)\, (I-Q(x)) u_{\pi/4}|| - r_n\\ &\geq& \frac{1}{2} \| T_\triangle(x, n, 0) u_{\beta(x)} \|   - r_n 
\\ &=& \frac{1}{2} \| T_\triangle(x, n, 0) u_{\beta(x)} \|\biggr(1   - \frac{2 r_n}{\| T_\triangle(x, n, 0) u_{\beta(x)} \|} \biggr).
\end{eqnarray*}
Note that, again,  
\begin{eqnarray*}
\lim_{n \to \infty} \frac{ r_n}{\| T_\triangle(x, n, 0) u_{\beta(x)} \|}  &\leq& \lim_{n \to \infty} r_n \|T_\triangle^{-1}(x, n, 0)\|\\ 
&=& \lim_{n \to \infty} r_n \|T_\triangle(x, n, 0)\| = 0    
\end{eqnarray*}
uniformly in $x \in [-2,2]$, since we assumed that $\tau > 5$.
Thus, there exists $n_0>1$ such that, for every $x \in [-2,2]$ and for all $n>n_0$,
\begin{eqnarray}\label{eq01010101678}
||T^\kappa (x, n, 0)u_{\pi/4}||^{-2} &\leq&  8 \| T_\triangle(x, n, 0) u_{\beta(x)} \|^{-2}.
\end{eqnarray}

Finally, let \( f \in C([-2,2]) \) be non-negative function and consider a continuous extension of \(f\) to $\mathbb{R}$, which we will also denote by \(f\). Again, let $\{\varphi_\epsilon\} \subset C_c([-2,2])$ be a sequence of positive continuous functions that converges pointwise to $\chi_{(-2,2)}$ as $\epsilon \to 0^+$, with $\varphi_\epsilon \equiv 0$ on $[-2,-2+\displaystyle\frac{\epsilon}{2}] \cup [2-\frac{\epsilon}{2},2]$,  $\varphi_\epsilon \equiv 1$ on $[-2+\epsilon,2-\epsilon]$, and $0\leq \varphi_\epsilon \leq 1$. Then, by Theorem \ref{CarmonaTHM}, \eqref{eq01010101678} and Lemma \ref{mainmainlema}, there exists $\gamma>0$ such that   
\begin{eqnarray}\label{eqeqeqeqeq010101016} \nonumber
\int f(x) \varphi_\epsilon(x)  \, d\mu(x) &=& \lim_{n \to \infty} \frac{1}{\pi} \int  f(x) \varphi_\epsilon(x) \left\| T^\kappa(x, n, 0)u_{\pi/4} \right\|^{-2} \, dx\\ \nonumber
&\leq& \gamma  \lim_{n \to \infty} \frac{1}{\pi} \int f(x) \varphi_\epsilon(x) \left\| T_\triangle(x, n, 0)u_\beta(x) \right\|^{-2} \, dx\\  
&=&  \gamma \int f(x) \varphi_\epsilon(x) \, \rho_{\beta(x)}(x) \, dx \leq  \gamma \int f(x) \varphi_\epsilon(x) \, d\mu^{\triangle}(x);
\end{eqnarray}
by Lemma \ref{lema-massa0} and dominated convergence, 
\begin{eqnarray*}
\int_{-2}^2 f(x) \, d\mu(x) \leq \gamma \int_{-2}^2 f(x)  \, d\mu^{\triangle}(x).
\end{eqnarray*}
Above in \eqref{eqeqeqeqeq010101016}, we have used that for every \( x \in (-2,2) \), one has
\[
\rho_{\beta(x)}(x) \leq \frac{\sqrt{|2 - x|}}{\pi \sqrt{|2 + x|}} + \frac{ \sqrt{|2 + x|}}{\pi \sqrt{|2 - x|}} = \frac{d\mu^{\triangle}(x)}{dx}
\]
(recall \eqref{derivadalaplacia}). This concludes the proof of the theorem.
\end{proof}

\subsubsection*{Proof of Main Theorem}
\ 

We proceed with the proof of Theorem \ref{app2}. We first prove the items $(ii)$ and $(iii)$. Next, we present a proof for \( H_{\pi/4}^\kappa \); since the proof for \( H_{+}^\kappa \) is identical, it is omitted. 

\ 

\noindent $(ii)$ As, for all $f \in  C([-2,2])$, 
\begin{equation}\label{eqlm0101}
    \mu_{\psi}^{\pi/4} = \mu_{f(H^\kappa)\delta_1}^{\pi/4} \leq  \left\| f\right\|_{\infty}^2  \mu_{\delta_1}^{\pi/4}, 
\end{equation}
it suffices to prove the result for $\mu = \mu_\kappa^{\pi/4} =  \mu_{\delta_1}^{\pi/4}$.

First, recall that  
\[
d\mu^{\triangle}(x) =  \frac{\sqrt{|2 - x|}}{\pi \sqrt{|2 + x|}} \cdot \chi_{(-2,2)}(x) \, dx + \frac{ \sqrt{|2 + x|}}{\pi \sqrt{|2 - x|}} \cdot \chi_{(-2,2)}(x) \, dx 
\]  
(see \eqref{derivadalaplacia}). Thus, for every interval \( I \) with \( |I| < 1 \),  
\begin{equation}\label{eq008}
\mu^{\triangle}(I) \leq 4 |I|^{1/2}.    
\end{equation}
Moreover, \( \mu^{\triangle} \) has a singularity of type \( \beta = \frac{1}{2} \).

Since Theorem~\ref{maintheorem0202} gives
\[
\int_{-2}^{2} f(x)\,d\mu(x)
\leq
\gamma \int_{-2}^{2} f(x)\,d\mu^{\triangle}(x)
\]
for every non-negative function \(f \in C([-2,2])\), it follows from the
standard characterization of the ordering of finite Borel measures that
\[
\left.\mu\right|_{[-2,2]} \leq \gamma\,\mu^{\triangle}.
\]
Therefore, for every Borel set \(\Lambda \subset [-2,2]\),
\begin{equation}\label{importantinequa}
\mu(\Lambda) \leq \gamma\,\mu^{\triangle}(\Lambda).
\end{equation}
Thus, by \eqref{eq008}, item~(ii) follows.

\ 

\noindent $(iii)$ Let \( f \in C([-2,2]) \) and let \( \mu_{\psi}^{\pi/4} = \mu_{f(H^\kappa)\delta_1}^{\pi/4} \). Since \( \mu^{\triangle} \) exhibits a singularity of type \( \beta = \frac{1}{2} \), it follows from Theorem~\ref{maintheorem0202} that \( \mu_{\delta_1}^{\pi/4} \) exhibits a weak-singularity of the same type. Therefore, by~\eqref{eqlm0101}, the measure \( \mu_{\psi}^{\pi/4} \) also exhibits a weak-singularity of type \( \beta = \frac{1}{2} \).
Thus, item $(iii)$ is a consequence of Theorem \ref{app1}, since by the Spectral Theorem we have
\[
\frac{1}{t}\int_0^t |\langle \psi, e^{-isH^\kappa} \psi \rangle|^2 \, ds = \frac{1}{t} \int_0^t \left| \int_{\mathbb{R}} e^{-isx} \, d\mu_\psi(x) \right|^2 ds.
\]

\ 

 \noindent $(i)$  Note that the canonical spectral
measure is purely absolutely continuous by~\eqref{importantinequa}. Since $\delta_1$ is a cyclic vector, this implies that the operator has purely absolutely continuous spectrum.
This concludes the proof of the theorem.


\section{Proofs of technical ingredients}\label{sec3}

\subsection{Proof of the Asymptotic formula I}

\ 

Let \(\{F_j\}_{j=1}^n\) be a sequence of square matrices. In this section, we adopt the following convention for the product:
\[
\prod_{j=0}^n F_j = F_n F_{n-1} \cdots F_0.
\]

In the proof of Theorem \ref{transferthm}, we will use the known telescoping formula for matrices given below. For the reader's convenience, we provide a proof in Appendix \ref{appendix}.

\begin{lemma}\label{01teclemma} For any $n \geq 1$,
\[
\prod_{j=0}^n F_j - \prod_{j=0}^n G_j = \sum_{k=0}^n \left( \prod_{j=k+1}^n F_j \cdot (F_k - G_k) \cdot \prod_{j=0}^{k-1} G_j  \right).
\]
\end{lemma}

\begin{lemma}\label{02teclemma}
Let $|\kappa| \leq 1/\|b\|_\infty$, and let \( T^\kappa(x, n, 0) \) be the transfer matrix of the operator $H^\kappa$ as in the statement of Theorem \ref{transferthm}, with $n \in J = \{ \mathbb{Z}, \mathbb{Z}^+ \}$. Then, there exists a constant $\gamma > 0$ such that for every $x \in \sigma(H^\kappa)$,
\[
\|T^\kappa(x, n, 0)\| \leq \gamma (a + 2\|V\|_\infty)^{|n|}, \, \,  n \in J.
\]
\end{lemma}

\begin{proof}In this particular proof, we employ the maximum norm for $2\times 2$ matrices, rather than the Frobenius norm. Since $|\kappa| \leq 1/\|b\|_\infty$, it follows that $\sigma(H^\kappa) \subset [-3 - \|V\|_\infty, 3 + \|V\|_\infty]$. Thus, for every $x \in \sigma(H^\kappa)$,
\[
|x| \leq 3 + \|V\|_\infty.
\]
Combining this inequality with the fact that $b(n) \to 0$ as $|n| \to  \infty$ and the assumption that $a > 3$,  the result follows.
\end{proof}

\begin{remark} {\rm Since \( H^\kappa = H \) for \( \kappa = 0 \), the conclusion of Lemma~\ref{02teclemma} applies to~\( H \).} 
\end{remark}

We proceed with the proof of Theorem \ref{transferthm}. Let $n \geq 1$ and let \( T^\kappa(x, n, 0) \) be the transfer matrix of the operator $H^\kappa$. We write
\[
T^\kappa(x, n, 0) = A_n A_{n-1} \cdots A_1 A_0 = \prod_{j=0}^n A_j.
\]
The transfer matrix of the operator $H$ is
\[
T(x, n, 0) = B_n B_{n-1} \cdots B_1 B_0  = \prod_{j=0}^n B_j.
\]

Also note that, for every $k \geq 1$,
\[
B_k-A_k =
\begin{pmatrix}
\kappa b(k) & 0 \\
0 & 0
\end{pmatrix}
\]
and $B_0-A_0=0$.

Note that, by Lemma \ref{01teclemma},
\begin{eqnarray*} 
&\,&\prod_{j=0}^n  B_j -  \prod_{j=0}^n A_j \\ 
&=&  \sum_{k=0}^n \left( \prod_{j=k+1}^n B_j  \cdot (B_k - A_k) \cdot \prod_{j=0}^{k-1}  A_j \right)\\ 
&=&  \sum_{k=0}^n \left( \biggr( \prod_{j=k+1}^n B_j \biggr)   \biggr(\prod_{j=0}^{k} B_j \biggr) \biggr(\prod_{j=0}^{k} B_j \biggr)^{-1}  \cdot (B_k - A_k) \cdot \prod_{j=0}^{k-1}  A_j \right) \\ 
&=&  \sum_{k=0}^n \left(\biggr( \prod_{j=0}^n B_j \biggr)   \biggr(\prod_{j=0}^{k} B_j \biggr)^{-1}  \cdot (B_k - A_k) \cdot \prod_{j=0}^{k-1}  A_j \right) \\ 
&=&  \biggr( \prod_{j=0}^n B_j \biggr) \sum_{k=0}^n \left(\biggr(\prod_{j=0}^{k} B_j \biggr)^{-1}  \cdot (B_k - A_k) \cdot \prod_{j=0}^{k-1}  A_j \right). 
\end{eqnarray*}

By Lemma \ref{02teclemma}, there exists a $\gamma'>0$ that does not depend on $x \in \sigma(H^\kappa)$, such that 

\begin{eqnarray*}
\nonumber 
&\,& \sum_{k=0}^n \biggr|\biggr| \left(    \biggr(\prod_{j=0}^{k} B_j \biggr)^{-1}  \cdot (B_k - A_k) \cdot \prod_{j=0}^{k-1}  A_j \right)  \biggr|\biggr| 
\\ \nonumber&\leq&  \sum_{k=0}^n \biggr|\biggr|\biggr(\prod_{j=0}^{k} B_j \biggr)^{-1}\biggr|\biggr| |\kappa| |b(k)| \prod_{j=0}^{k-1} ||A_j||\\ 
\nonumber  
&=&  \sum_{k=0}^n \biggr|\biggr|\biggr(\prod_{j=0}^{k} B_j \biggr)\biggr|\biggr| |\kappa| |b(k)| \prod_{j=0}^{k-1} ||A_j||   \\ 
\nonumber 
&\leq&  \sum_{k=0}^n  |\kappa| |b(k)| \prod_{j=0}^{k-1} ||A_j||   \prod_{j=0}^{k} ||B_j|| \\ 
\nonumber  
&\leq& |\kappa| \gamma' \sum_{k=0}^n  |b(k)| (a+2||V||_\infty)^{2k}   \\ 
&\leq& \gamma' \gamma |\kappa| \sum_{k=0}^{\infty} (a+2||V||_\infty)^{(2-\tau)k}.
\end{eqnarray*} 
We used above that, if \(\det(B) = 1\), then \(\|B\| = \|B^{-1}\|\). Hence, we can rewrite
\begin{eqnarray}\label{maineq0202} 
\prod_{j=0}^n  B_j -  \prod_{j=0}^n A_j  
&=&  \biggr( \prod_{j=0}^n B_j \biggr) (Q-R_n), 
\end{eqnarray}
where
\[
Q = \sum_{k=0}^\infty \left(\biggr(\prod_{j=0}^{k} B_j \biggr)^{-1}  \cdot (B_k - A_k) \cdot \prod_{j=0}^{k-1}  A_j \right),
\]
\[
R_n = \sum_{k=n+1}^\infty \left(\biggr(\prod_{j=0}^{k} B_j \biggr)^{-1}  \cdot (B_k - A_k) \cdot \prod_{j=0}^{k-1}  A_j \right),
\]
with $||R_n|| \leq \eta |\kappa| (a+2||V||_\infty)^{(2-\tau)n}$  and $||Q|| \leq  \eta |\kappa|$ for some $\eta>0$, which does not depend on $x \in \sigma(H^\kappa)$.

As, by Lemma \ref{02teclemma}, there exists an $\eta'>0$ such that 
\[
\biggr|\biggr| \biggr(\prod_{j=0}^n B_j \biggr) R_n \biggr|\biggr|  \leq \eta' (a+2||V||_\infty)^{(3-\tau)n},
\]
one has 
\begin{eqnarray*}
T(x, n, 0) - T^\kappa(x, n, 0)  = T(x, n, 0)Q + R_n^{'},
\end{eqnarray*}
with $||R_n^{'}|| \leq \eta' |\kappa| (a+2||V||_\infty)^{(3-\tau)n}$. 

Note that we can apply the same above argument to
\[
\prod_{j=0}^n  A_j -  \prod_{j=0}^n B_j,
\]
and obtain an identity similar to \eqref{maineq0202}. Hence, the result follows for $n \geq 1$. The proof for \( n \leq -1 \) is identical and is therefore omitted.

Finally, note that \( Q(x) = \sum_{k=0}^\infty q_k(x) \), where  
\[
q_k(x) =  \left( \prod_{j=0}^{k} B_j(x) \right)^{-1} 
\begin{pmatrix}
\kappa b(k) & 0 \\
0 & 0
\end{pmatrix} \prod_{j=0}^{k-1} A_j(x)
\]
is a polynomial in the variable \( x \). Since the convergence \( Q(x) = \displaystyle\lim_{n \to \infty} \sum_{k=0}^n q_k(x) \) is uniform in \( x \in \sigma(H^\kappa) \), it follows that \( Q \) is continuous on \( \sigma(H^\kappa) \).  This concludes the proof of the theorem. 
 \hfill \qedsymbol

\noindent {\bf  Proof of Theorem \ref{prop:asymptotic-formula-polynomial}}

The main ingredient in the proof below is an idea due to Damanik et al. \cite{Damanik4}, based on combining a telescoping argument with a discrete Gronwall inequality.

Throughout this subsection, $T_\triangle(x,n,0)$ denotes the transfer matrix of the free Laplacian and $T^\kappa(x,n,0)$ denotes the transfer
matrix associated with
\[
H^\kappa=\triangle+\kappa b
\]
on $\ell^2(\mathbb Z^+)$.

The next lemma gives the corresponding a priori estimate for the
perturbed transfer matrices.

\begin{lemma}\label{lem:perturbed-linear-growth-sch}
Let $b\in\ell^\infty(\mathbb Z^+,\mathbb R)$ satisfy
\[
M_1(b):=
\sum_{n=1}^{\infty}(n+1)|b(n)|<\infty.
\]
Then, there exists $C>0$ such that
\[
\|T^\kappa(x,n,0)\|
+
\|T^\kappa(x,n,0)^{-1}\|
\leq C(n+1),
\, \, 
n\geq0,\, \,  x\in[-2,2],
\]
where $C$ can be chosen uniformly for $\kappa$ in a sufficiently small
neighborhood of the origin.
\end{lemma}

\begin{proof}
Set
\[
B_j(x)=
\begin{pmatrix}
x & -1\\
1 & 0
\end{pmatrix},
\, \, 
C_j^\kappa(x)=
\begin{pmatrix}
x-\kappa b(j) & -1\\
1 & 0
\end{pmatrix}.
\]
Then
\[
\|C_j^\kappa(x)-B_j(x)\|
\leq |\kappa|\,|b(j)|,
\, \, 
j\geq1,\, \,  x\in[-2,2].
\]

Using the telescoping identity, we obtain
\[
T^\kappa(x,n,0)
=
T_\triangle(x,n,0)
+
\sum_{j=1}^{n}
T_\triangle(x,n,0)
T_\triangle(x,j,0)^{-1}
\bigl(C_j^\kappa(x)-B_j(x)\bigr)
T^\kappa(x,j-1,0).
\]
By Lemma~\ref{lem:free-linear-growth},
\[
\|T^\kappa(x,n,0)\|
\leq
C_0(n+1)
+
C_1|\kappa|
\sum_{j=1}^{n}
(n-j+1)|b(j)|
\|T^\kappa(x,j-1,0)\|.
\]
Define
\[
X_n(x):=
\frac{\|T^\kappa(x,n,0)\|}{n+1}.
\]
Since
\[
\frac{n-j+1}{n+1}\leq1,
\]
we obtain
\[
X_n(x)
\leq
C_0
+
C_1|\kappa|
\sum_{j=1}^{n}
j|b(j)|X_{j-1}(x).
\]
The discrete Gronwall inequality therefore yields
\[
X_n(x)
\leq
C_0
\exp\left(
C_1|\kappa|
\sum_{j=1}^{\infty}j|b(j)|
\right).
\]
Hence,
\[
\|T^\kappa(x,n,0)\|
\leq C(n+1),
\]
uniformly in $n\geq0$ and $x\in[-2,2]$, with $C$ uniform for $\kappa$
in a sufficiently small neighborhood of the origin.

Finally, since the transfer matrices are unimodular,
\[
\|T^\kappa(x,n,0)^{-1}\|
=
\|T^\kappa(x,n,0)\|,
\]
and the result follows.
\end{proof}

\subsubsection*{ Proof of Theorem \ref{prop:asymptotic-formula-polynomial}} Set
\[
B_j(x)=
\begin{pmatrix}
x & -1\\
1 & 0
\end{pmatrix},
\, \, 
C_j^\kappa(x)=
\begin{pmatrix}
x-\kappa b(j) & -1\\
1 & 0
\end{pmatrix}.
\]
Then
\begin{equation}\label{eq:step-difference-poly-sch}
\|B_j(x)-C_j^\kappa(x)\|
\leq
|\kappa|\,|b(j)|,
\, \, 
j\geq1,\, \,  x\in[-2,2].
\end{equation}

The telescoping identity gives
\[
T_\triangle(x,n,0)-T^\kappa(x,n,0)
=
T_\triangle(x,n,0)
\sum_{j=1}^{n}q_j(x),
\]
where
\[
q_j(x)
=
T_\triangle(x,j,0)^{-1}
\bigl(B_j(x)-C_j^\kappa(x)\bigr)
T^\kappa(x,j-1,0).
\]
By Lemmas~\ref{lem:free-linear-growth} and
\ref{lem:perturbed-linear-growth-sch}, together with
\eqref{eq:step-difference-poly-sch},
\begin{align*}
\|q_j(x)\|
&\leq
\|T_\triangle(x,j,0)^{-1}\|
\|B_j(x)-C_j^\kappa(x)\|
\|T^\kappa(x,j-1,0)\|\\
&\leq
C|\kappa|(j+1)j|b(j)|\\
&\leq
C|\kappa|(j+1)^2|b(j)|.
\end{align*}
Therefore, 
\[
Q(x):=
\sum_{j=1}^{\infty}q_j(x)
\]
converges absolutely and uniformly for $x\in[-2,2]$. In particular,
\[
\|Q(x)\|
\leq
C|\kappa|
\sum_{j=1}^{\infty}(j+1)^2|b(j)|
=
C|\kappa|M_2(b).
\]
Since every $q_j$ is continuous in $x$, the uniform convergence also
implies that $Q$ is continuous on $[-2,2]$.

Set
\[
E_n(x):=
\sum_{j=n+1}^{\infty}q_j(x).
\]
Then
\[
T_\triangle(x,n,0)-T^\kappa(x,n,0)
=
T_\triangle(x,n,0)Q(x)
-
T_\triangle(x,n,0)E_n(x).
\]
Thus, setting
\[
R_n(x):=-T_\triangle(x,n,0)E_n(x),
\]
Lemma~\ref{lem:free-linear-growth} gives
\begin{align*}
\|R_n(x)\|
&\leq
\|T_\triangle(x,n,0)\|
\sum_{j=n+1}^{\infty}\|q_j(x)\|\\
&\leq
C|\kappa|(n+1)
\sum_{j=n+1}^{\infty}(j+1)^2|b(j)|.
\end{align*}
This proves \eqref{eq:poly-comparison-sch-I} and the required estimate
for $R_n$.

The second identity is obtained in exactly the same way by interchanging
the roles of $T_\triangle$ and $T^\kappa$. Namely, one defines
\[
\widetilde q_j(x)
=
T^\kappa(x,j,0)^{-1}
\bigl(C_j^\kappa(x)-B_j(x)\bigr)
T_\triangle(x,j-1,0),
\]
and
\[
P(x):=
\sum_{j=1}^{\infty}\widetilde q_j(x).
\]
The same estimates give
\[
\|P(x)\|
\leq C|\kappa|M_2(b)
\]
and
\[
\|K_n(x)\|
\leq
C|\kappa|(n+1)
\sum_{j=n+1}^{\infty}(j+1)^2|b(j)|.
\]
This completes the proof of the theorem.
\hfill \qedsymbol

\subsection{Proof of the Strichartz-type inequality}
\ 

The main ingredient of this proof is a refined analysis of the singularity using Fourier analysis and Cauchy's theory. Indeed, the core of the proof consists of two lemmas (Lemmas \ref{1teclemma} and \ref{mainlemma}), which are stated and proved below. We note that the following discussion sharpens some of the arguments presented by the authors in \cite{Aloisio}.

\begin{lemma}\label{1teclemma} Let $g \in {\mathrm L}^1(\mathbb{R})$. For each $t>0$, one has
\begin{eqnarray}\nonumber \int_{\mathbb{R}} e^{-\pi t^2|x-y|^2} g(x) \, dx  &=& \frac{1}{t} \int_{\mathbb{R}}  e^{2\pi i y \xi} e^{- \frac{\pi |\xi|^2}{t^2}}  \widehat{g}(\xi) \, d\xi.
\end{eqnarray}
\end{lemma}

\begin{proof} Let $g \in {\mathrm L}^1(\mathbb{R})$ and let $(g_n)_{n \geq 1} \subset {\mathrm L}^1(\mathbb{R}) \cap {\mathrm L}^2(\mathbb{R})$ be such that $\displaystyle\lim_{n \to \infty}\|g_n - g\|_{{\mathrm L}^1(\mathbb{R})} = 0$. Then, it follows from Theorem 4.9 in \cite{Brezis} that there exists a subsequence $(g_{n_k})$ and a function $h \in {\mathrm L}^1(\mathbb{R})$ such that $\displaystyle\lim_{k \to \infty} g_{n_k}(x) = g(x)$ for almost every $x \in \mathbb{R}$, and for every $k \geq 1$, $ |g_{n_k}(x)| \leq h(x)$ for  almost every $x \in \mathbb{R}$. We note that for each $t >0$, each $k \geq 1$ and each $\xi \in \mathbb{R}$, 
\[  e^{- \frac{\pi |\xi|^2}{t^2}} |\widehat{g_{n_k}}(\xi)|  \leq  e^{- \frac{\pi |\xi|^2}{t^2}} \|g_{n_k}\|_{{\mathrm L}^1(\mathbb{R})}  \leq  e^{- \frac{\pi |\xi|^2}{t^2}} \|h\|_{{\mathrm L}^1(\mathbb{R})},\]
and 
\[\int_{\mathbb{R}} e^{- \frac{\pi |\xi|^2}{t^2}} d\xi = t.\]
This shows that, for every $t>0$, the sequence $e^{- \frac{\pi |\xi|^2}{t^2}} |\widehat{g_{n_k}}(\xi)|$ is dominated by an integrable function.  
 
Set, for each $t>0$ and each $x \in \mathbb{R}$, $\Phi_t(x) := e^{-\pi t|x|^2}$. Then, for each $t>0$,
\begin{equation}\label{eq0lemma}
\widehat{\Phi_t}(\xi) = \frac{1}{t} e^{- \frac{\pi |\xi|^2}{t^2}}, \, \,  \xi \in \mathbb{R}.
\end{equation}

It follows from relation (\ref{eq0lemma}), some properties of the Fourier transform, dominated convergence and Plancherel's Theorem that for each $y \in \mathbb{R}$ and each $t >0$, 
\begin{eqnarray*}\nonumber  \int_{\mathbb{R}} e^{-\pi t^2|x-y|^2} g(x) \, dx  &=& \lim_{k \to \infty} \int_{\mathbb{R}} e^{-\pi t^2|x-y|^2} g_{n_k} (x)  \, dx\\ \nonumber &=&  \lim_{k \to \infty} \int_{\mathbb{R}} \overline{(\tau_y \Phi_t)(x)} g_{n_k} (x)\, dx \\ \nonumber  &=& \lim_{k \to \infty} \int_{\mathbb{R}}  \overline{\widehat{(\tau_y \Phi_t)}(\xi)} \widehat{g_{n_k}}(\xi) \,  d\xi \\ \nonumber &=& \lim_{k \to \infty}  \int_{\mathbb{R}} e^{2\pi i y \xi} \widehat{\Phi_t}(\xi)      \widehat{g_{n_k}}(\xi) \, d\xi  \\    \nonumber &=& \lim_{k \to \infty} \frac{1}{t} \int_{\mathbb{R}} e^{2\pi i y \xi} e^{- \frac{\pi |\xi|^2}{t^2}} \widehat{g_{n_k}}(\xi)  \, d\xi\\ &=& \frac{1}{t} \int_{\mathbb{R}}  e^{2\pi i y \xi} e^{- \frac{\pi |\xi|^2}{t^2}}  \widehat{g}(\xi) \, d\xi,
\end{eqnarray*}
where $\tau_yg(\cdot) = g(\cdot-y)$ stands for the translation by $y\in\mathbb{R}$.
\end{proof}

\begin{lemma}\label{harmoniclemma2}
Let $\beta \in (0,1)$, $a \in \mathbb{R}$, and $c > 0$. Set, for every $x \in \mathbb{R}$,
\[
f(x) = \frac{1}{|x - a|^\beta} \, \chi_{[-c,c]}(x).
\]
Then, for every $\xi > 0$, the Fourier transform of $f$ is given by
\[
\widehat{f}(\xi) = \frac{e^{-2\pi i a \xi}}{(2\pi)^{1 - \beta} \, \xi^{1 - \beta}} \int_{2\pi \xi (-c - a)}^{2\pi \xi (c - a)} e^{-i v} |v|^{-\beta} \, dv.
\]
In particular, there exists $\gamma > 0$ such that, for all $\xi \in \mathbb{R}$, $|\widehat{f}(\xi)| \leq \gamma |\xi|^{\beta - 1}.$
\end{lemma}

\begin{proof}
We first apply the change of variable $u = x - a$. Then, the Fourier transform becomes
\[
\widehat{f}(\xi) = \int_{-c}^c \frac{e^{-2\pi i x \xi}}{|x - a|^\beta} \, dx = e^{-2\pi i a \xi} \int_{-c - a}^{c - a} \frac{e^{-2\pi i u \xi}}{|u|^\beta} \, du.
\]
Now, we set $v = 2\pi \xi u$, hence $u = \frac{v}{2\pi \xi}$ and $du = \frac{dv}{2\pi \xi}$. Substituting, we obtain
\[
\widehat{f}(\xi) = e^{-2\pi i a \xi} (2\pi \xi)^{\beta - 1} \int_{2\pi \xi (-c - a)}^{2\pi \xi (c - a)} e^{-i v} |v|^{-\beta} \, dv.
\]
 Note that, as $\widehat{f}(-\xi)=\overline{\widehat{f}(\xi)}$, it follows that 
\[
\limsup_{|\xi| \to \infty} |\xi|^{1-\beta} |\widehat{f}(\xi)| \leq \left| \int_{-\infty}^{\infty} e^{-i v} |v|^{-\beta} \, dv \right| = \Gamma(1 - \beta),
\]
where $\Gamma$ denotes the standard gamma function. Since for all $\xi \in \mathbb{R}$, 
\[|\widehat{f}(\xi)| \leq \|f\|_{{\mathrm L}^1(\mathbb{R})},\]
the result follows.
\end{proof}

\begin{lemma}\label{mainlemma} Let \( f \) and \( g \) be non-negative real-valued functions. Suppose that there exist \( a, b \in \mathbb{R} \), \( c,\gamma, \gamma' > 0 \) and \( \varrho \geq 0 \) such that for every \( x \in \mathbb{R} \),
\[
f(x) \leq \frac{\gamma}{|x - a|^\varrho} \chi_{[-c,c]}(x) =: f_\varrho(x) \, \,  \text{and} \, \,  g(x) \leq \frac{\gamma'}{|x - b|^\varrho} \chi_{[-c,c]}(x) =: g_\varrho(x).
\]
Then, there exists a constant $M = M(c,\gamma,\gamma',\varrho)$ such that:

\begin{enumerate}

\item[(i)]  if $0\leq \varrho < \frac{1}{2}$, then for every $t>0$, 
\begin{eqnarray*}
 \int_{\mathbb{R}} \int_{\mathbb{R}} e^{-\pi t^2|x-y|^2} f(x) g(y)   \, dx dy  \; \leq M \, \frac{1}{t};
\end{eqnarray*}

\item[(ii)]  if $\varrho = \frac{1}{2}$, then for every $t>2$, 
\begin{eqnarray*}
 \int_{\mathbb{R}} \int_{\mathbb{R}} e^{-\pi t^2|x-y|^2} f(x) g (y)   \, dx dy  \; \leq M \, \frac{\log(t)}{t};
\end{eqnarray*}

\item[(iii)] if $\frac{1}{2} < \varrho < 1$, then for every $t>0$, 
\begin{eqnarray*}
 \int_{\mathbb{R}} \int_{\mathbb{R}} e^{-\pi t^2|x-y|^2} f(x) g (y)  \, dx dy  \; \leq M  \, \frac{1}{t^{2(1-\varrho)}}.
\end{eqnarray*}

\end{enumerate}
\end{lemma}

\begin{proof}   By taking $g=f_\varrho(\cdot)$ in Lemma \ref{1teclemma}, one has

\begin{eqnarray*}\label{eq00lemma} \int_{\mathbb{R}} e^{-\pi t^2|x-y|^2} f(x) \, dx   &\leq&  \gamma \int_{\mathbb{R}} e^{-\pi t^2|x-y|^2} f_\varrho(x) \, dx\\       &=& \frac{\gamma}{t} \int_{\mathbb{R}} e^{2\pi i y \xi} e^{- \frac{\pi |\xi|^2}{t^2}}  \widehat{f_\varrho}(\xi) \, d\xi.    
\end{eqnarray*}
Thus, by Fubini's Theorem, one obtains for each $t >0$,   

\begin{eqnarray}\label{eq1lemma} \nonumber  &\,& \int_{\mathbb{R}} \int_{\mathbb{R}} e^{-\pi t^2|x-y|^2} f(x)  g(y)\, dx dy\\ \nonumber    &=&  \int_{\mathbb{R}} \biggr(\int_{\mathbb{R}} e^{-\pi t^2|x-y|^2} f(x)  \, dx \biggr) g(y) dy\\ &\leq& \nonumber \frac{\gamma}{t}  \int_{\mathbb{R}} \biggr(\int_{\mathbb{R}} e^{2\pi i y \xi} e^{- \frac{\pi |\xi|^2}{t^2}}  \widehat{f_\varrho}(\xi) \, d\xi \biggr) g(y) dy\\ &\leq&  \nonumber \frac{\gamma \gamma'}{t}  \int_{\mathbb{R}} \biggr(\int_{\mathbb{R}} e^{2\pi i y \xi} e^{- \frac{\pi |\xi|^2}{t^2}}  \widehat{f_\varrho}(\xi) \, d\xi\biggr) g_{\varrho}(y) dy\\ \nonumber &=& \frac{\gamma \gamma'}{t} \int_{\mathbb{R}} e^{- \frac{\pi |\xi|^2}{t^2}}  \widehat{f_\varrho}(\xi) \biggr( \int_{\mathbb{R}}  e^{2\pi i y \xi} g_{\varrho}(y) \, dy \biggr) \, d\xi\\   &=& \frac{\gamma \gamma'}{t} \int_{\mathbb{R}} e^{- \frac{\pi |\xi|^2}{t^2}}  \widehat{f_\varrho}(\xi)  \overline{\widehat{g_{\varrho}}}(\xi) \, d\xi. \\ \nonumber
\end{eqnarray}

Now a separate argument is  necessary for each item.

\

\noindent $(i)$ Given that, for every \( 0 \leq \varrho < \frac{1}{2} \), \( f_\varrho, g_\varrho \in {\mathrm L}^2(\mathbb{R}) \), the result follows from \eqref{eq1lemma}, combined with the Cauchy-Schwarz inequality and Plancherel's Theorem.

\ 

\noindent $(ii)$ $\varrho=1/2$. To every $t> 0$,
\begin{eqnarray*} \int_{-1}^1 e^{- \frac{\pi |\xi|^2}{t^2}}  \widehat{f_\varrho}(\xi)  \overline{\widehat{g_{\varrho}}}(\xi) \, d\xi   &\leq& 2||f_\varrho||_{{\mathrm L}^1(\mathbb{R})} ||g_\varrho||_{{\mathrm L}^1(\mathbb{R})}.
\end{eqnarray*}
By Lemma \ref{harmoniclemma2}, there exists $\gamma'' > 0$ such that, for every $t > 0$,
\[\int_1^\infty e^{- \frac{\pi |\xi|^2}{t^2}}  |\widehat{f_\varrho}(\xi)|  |\widehat{g_{\varrho}}(\xi)| \, d\xi \leq \gamma''\int_1^\infty \frac{e^{- \frac{\pi |\xi|^2}{t^2}} }{|\xi|}  \, d\xi.  \]
Since, by Cauchy's Residue Theorem, for every $t >0$
\begin{equation}\label{eq3lemma}
\int_1^\infty \frac{e^{- \frac{\pi |\xi|^2}{t^2}} }{|\xi|}  \, d\xi = \frac{\Gamma(0, \pi/t^2)}{2} \sim \log(t),    
\end{equation}
one gets from (\ref{eq1lemma}) and (\ref{eq3lemma}) that for every $t >2$,
\begin{eqnarray*}
 \int_{\mathbb{R}} \int_{\mathbb{R}} e^{-\pi t^2|x-y|^2} f(x) g(y)  \, dx dy     &\leq&  \frac{\gamma \gamma'}{t} \int_{-1}^1  e^{- \frac{\pi |\xi|^2}{t^2}}   |\widehat{f_\varrho}(\xi)|  |\widehat{g_{\varrho}}(\xi)|    \, d\xi\\  &+& \frac{\gamma \gamma'}{t} \int_{|x|> 1}  e^{- \frac{\pi |\xi|^2}{t^2}}  |\widehat{f_\varrho}(\xi)|  |\widehat{g_{\varrho}}(\xi)|   \, d\xi\\  &\leq&   \frac{2||f_\varrho||_{{\mathrm L}^1(\mathbb{R})} ||g_\varrho||_{{\mathrm L}^1(\mathbb{R})}  \gamma \gamma'}{t}\\ &+&   \frac{ \gamma \gamma' \gamma'' \Gamma(0, \pi/t^2)}{t} \lesssim \frac{\log(t)}{t},
\end{eqnarray*} 
where $\Gamma(0, \cdot)$ standard for the incomplete gamma function \cite{Abramowitz}.

\ 

\noindent $(iii)$ $1/2<\varrho<1$. Since it follows from Cauchy's Residue Theorem that for every $t >0$,  

\begin{eqnarray}\label{eq2lemma}\nonumber
   \int_{\mathbb{R}} \frac{e^{- \frac{\pi |\xi|^2}{t^2}} }{|\xi|^{2(1-\varrho)}  }  \, d\xi &=&   \pi^{1/2-\varrho}\Gamma\biggr(\varrho - \frac{1}{2}\biggr) t^{2\varrho-1},
\end{eqnarray}
one gets from Lemma \ref{harmoniclemma2}, (\ref{eq1lemma}) and (\ref{eq2lemma}) that for every $t >0$,  
\begin{eqnarray*}
\int_{\mathbb{R}} \int_{\mathbb{R}} e^{-\pi t^2|x-y|^2} f(x)  g(y)\, dx dy \lesssim   t^{-2(1-\varrho)}.
\end{eqnarray*}
\end{proof}

Let us proceed to the proof of Theorem \ref{app1} $(ii)$. As $e^{4\pi^2-(2\pi s)^2/t^2} \geq 1$ for  $s \in [0,t]$ and, by dominated convergence, for every $t>0,$
\[G(y) = \int_{\mathbb{R}}  e^{-\frac{t^2|x-y|^2}{4}} d\mu(x)\]
is continuous, it follows from Lemma \ref{mainlemma} $(ii)$ that, for every $t>2$,
\begin{eqnarray*}\label{maineq1}
\nonumber &\, &\frac{1}{t} \int_0^t \bigg|\int_{\mathbb{R}} e^{-2\pi isx} d\mu (x) \bigg|^2  ds\\ &\leq& \frac{1}{t} \int_0^t \bigg|\int_{\mathbb{R}} e^{-2\pi isx} d\mu (x) \bigg|^2 \, e^{4\pi^2-(2\pi s)^2/t^2} \, ds\\ \nonumber &\leq& \frac{e^{4\pi^2}}{t} \int_{-\infty}^{\infty} \bigg|\int_{\mathbb{R}} e^{-2\pi isx} d\mu (x) \bigg|^2 \, e^{-(2\pi s)^2/t^2} \, ds\\ \nonumber &=& \frac{e^{4\pi^2} }{t}  \int_{\mathbb{R}} \int_{\mathbb{R}}  \biggl\{ \int_{-\infty}^{\infty} \,  e^{-((2\pi s)^2/t^2)-2\pi is(x-y)}  \, ds \biggl\} d\mu (x) d\mu (y) \\ \nonumber &=&  \frac{e^{4\pi^2}  \sqrt{\pi}}{2 \pi}  \int_{\mathbb{R}} \int_{\mathbb{R}}  e^{-\frac{t^2|x-y|^2}{4}} d\mu(x) d\mu(y)\\ \nonumber &\leq&  \frac{ \gamma^2  e^{4\pi^2}  \sqrt{\pi}}{2 \pi}  \int_{\mathbb{R}} \int_{\mathbb{R}}  e^{-\frac{t^2|x-y|^2}{4}} d\nu(x) d\nu(y)\\ \nonumber  &\leq& \sum_{j=1}^n \sum_{l=1}^n \frac{\gamma^2  e^{4\pi^2}  \sqrt{\pi}}{2 \pi}  \int_{\mathbb{R}} \int_{\mathbb{R}}    e^{-\pi (t/2\sqrt{\pi})^2|x-y|^2} f_j(x) f_l(y) dx dy \\ &\leq&  C_\mu \cdot\frac{\log(t)}{t}. 
\end{eqnarray*}

This completes the proof of Theorem \ref{app1} $(ii)$. The proofs of parts~$(i)$ and~$(iii)$ are analogous and are therefore omitted. \hfill \qedsymbol


\section{Proof of Theorem \ref{app3}}\label{sec4}

\subsection{Asymptotic formula II}
\ 

A few preparatory steps are necessary before proving Theorem~\ref{app3}. In what follows, we consider the Schr\"odinger operator \( H = \triangle + V, \) with \( V \in \ell^\infty(\mathbb{Z},\mathbb{R})\). 

We recall that the Weyl-Titchmarsh \(m^+\)-function (\(m^-\)-function), which is the Borel transform of the spectral measure \(\mu_+ = \mu_{\delta_1}^+\) (\(\mu_- = \mu_{\delta_0}^-\)) of the corresponding half-line problem with Dirichlet boundary conditions, is given by
\[
m^\pm(z) = \int \frac{d\mu_\pm(x)}{x - z},\, \,  z\in\mathbb{H},
\]
where $\mathbb{H} = \{z\in\mathbb{C}\mid \Im z > 0\}$. Thus, \(m^\pm\) has positive imaginary part for every \(\epsilon > 0\). Let, for each $E\in\mathbb{R}$ and each $\epsilon>0$,
\[
M(E + i\epsilon) = \int \frac{1}{x - (E + i\epsilon)}d\mu(x),
\]
where $\mu=\mu_{\delta_0}+\mu_{\delta_1}$. Note that \(M(E + i\epsilon) \in \mathbb{H}\). It is known that for every $\epsilon>0$,
\begin{equation}\label{eq0000000000000000}
    \Im M(E + i\epsilon) \geq \frac{1}{2\epsilon} \mu(E - \epsilon, E + \epsilon).
\end{equation}

We also recall that \cite{Jitomirskaya2}
\[m^+_\beta = R_{-\beta/2\pi} \cdot m^+.\]
Those are Borel transforms of the spectral measures \(\mu^\beta = \mu_{\delta_1}^\beta\) of \(H_\beta\) on \(\ell^2(\mathbb{Z}^{+})\) with boundary conditions \(u(0) \cos(\beta) + u(1) \sin(\beta) = 0\), $-\frac{\pi}{2} < \beta <\frac{\pi}{2}$. Here, 
\[
\begin{pmatrix} a & b \\
c & d \end{pmatrix} \cdot z = \frac{az + b}{cz + d}
\]
stands for the action of \(\mathrm{SL}(2, \mathbb{C})\) on \(\mathbb{C}\).

Let $\psi:\mathbb{H}\rightarrow\mathbb{R}_+$, \(\psi(z) = \sup_\beta |z^\beta|\). It was shown in \cite{AvilaUaH} that
\begin{equation}\label{maineqeq1}
\psi(z)^{-1} \leq \Im z \leq |z| \leq \psi(z),    
\end{equation}
\begin{equation}\label{maineqeq2}
\psi(M) \leq \psi(m_+)
\end{equation}
(regardless of the value of \(m_- \in \mathbb{H}\)). 

Set, for each $k \geq 1$ and each $E \in \mathbb{R}$,
\begin{equation}\label{operators}
P_{k}(E) = \sum_{n=1}^k T^*(E, 2n-1, 0)T(E, 2n-1, 0).    
\end{equation}

The following result by Avila and Jitomirskaya is a version of the Jitomirskaya-Last inequality (Theorem 1.1 in \cite{Jitomirskaya}, see also \eqref{AJeq01} below) for \(P_k(E)\). We present the details, for the sake of clarity.
\begin{lemma}[Lemma 4.2. in \cite{AvilaUaH}]\label{lemmaAvila} Let $E\in\mathbb{R}$, $k \geq 1$ and let $P_k(E)$ be given by~\eqref{operators}. Then, there exists a constant $\gamma>0$ such that
\[
\gamma^{-1} < \frac{ \psi(m^{+}\left(E + i\epsilon_k \right)) }{||P_k(E)||\epsilon_k}   < \gamma,
\]
where $\eps_k:=\dfrac{1}{\sqrt{4\det P_{k}(E)}}$.
\end{lemma}

\begin{proof}
Set, for every $L \geq 1$,
\[||u||_L^2 := \sum_{n=1}^L |u(n)|^2, \, u \in \ell^2(\mathbb{Z}).\]
As we deal only with bounded potentials, this definition of norm does not affect the inequality in Theorem 1.1 in \cite{Jitomirskaya}. 

Let $u^\beta$, $\beta\in(-\pi/2,\pi/2]$, be a solution to the eigenvalue equation $H_\beta u=Eu$ such that 
\[u^\beta(0) \cos(\beta) + u^\beta(1) \sin(\beta) = 0,\]
\[\, \,  |u^\beta(0)|^2 + |u^\beta(1)|^2 = 1.
\]
Let $0<\epsilon:=\dfrac{1}{2\|u^\beta\|_L \|u^{\beta + \pi/2}\|_L}$. It follows from Theorem 1.1 in \cite{Jitomirskaya} (see also (2.13) in \cite{Jitomirskaya2}) that 
\begin{equation}\label{AJeq01}
5 - \sqrt{24} < |m^+_\beta(E + i\epsilon)| \frac{\|u^\beta\|_L}{\|u^{\beta + \pi/2}\|_L} < 5 + \sqrt{24},    
\end{equation}
or, equivalently, that
\[
5 - \sqrt{24} < \frac{|m^+_\beta(E + i\epsilon)|}{2\epsilon \|u^{\beta + \pi/2}\|_L^2} < 5 + \sqrt{24}.
\]

Note that for any $\beta\in(-\pi/2,\pi/2]$, if $L=2k$, then
\begin{eqnarray*}\nonumber \|u^\beta\|_{L}^2 &=&\|u^\beta\|_{2k}^2 = \sum_{n=1}^{k} |u^\beta(2n)|^2 + \sum_{n=1}^{k} |u^\beta(2n-1)|^2 
\\ &=& \sum_{n=1}^k \left\langle 
\begin{pmatrix}
u^\beta(2n) \\ 
u^\beta(2n-1)
\end{pmatrix},
\begin{pmatrix}
u^\beta(2n) \\ 
u^\beta(2n-1)
\end{pmatrix} 
\right\rangle\\ \nonumber
&=& \sum_{n=1}^k \left\langle T(E, 2n-1, 0) 
\begin{pmatrix}
u^\beta(1) \\ 
u^\beta(0)
\end{pmatrix},
T(E, 2n-1, 0)
\begin{pmatrix}
u^\beta(1) \\ 
u^\beta(0)
\end{pmatrix} 
\right\rangle\\ \nonumber &=& \left\langle \sum_{n=1}^k  T^*(E, 2n-1, 0)T(E, 2n-1, 0) 
\begin{pmatrix}
u^\beta(1) \\ 
u^\beta(0)
\end{pmatrix},
\begin{pmatrix}
u^\beta(1) \\ 
u^\beta(0)
\end{pmatrix} 
\right\rangle\\ &\leq& \biggr \|\sum_{n=1}^k T^*(E, 2n-1, 0)T(E, 2n -1, 0) \biggr\|\\ &=& ||P_k(E)||,
\end{eqnarray*}
with equality attained for $\beta$ that maximizes $\|u^\beta\|_{L}^2$. Thus,
\begin{equation}\label{detPk}
\det P_k(E) = \inf_\beta \|u^\beta\|_L^2 \|u^{\beta + \pi/2}\|_L^2,    
\end{equation}
the infimum being attained at the critical points of \(\beta \mapsto \|u^\beta\|_L^2\). Therefore, if  $\eps_k=\dfrac{1}{\sqrt{4\det P_{k}(E)}}$, then for every \(\beta\in(-\pi/2,\pi/2]\),
\[
\frac{|m^+_\beta(E + i\epsilon_k)|}{2\epsilon_k \|P_k(E)\|} < 5 + \sqrt{24},
\]
and if \(\beta\) is such that \(\|u^{\beta + \pi/2}\|_L^2\) is maximal, then
\[
\frac{|m^+_\beta(E + i\epsilon_k)|}{2\epsilon_k \|P_k(E)\|} > 5 - \sqrt{24}.
\]
\end{proof}

The following result provides the final ingredient needed for the proof of Theorem~\ref{app3}.

\begin{theorem}\label{transferthm2} Let $H= \triangle+V$ be acting on  \( \ell^2(\mathbb{Z}) \), with $V \in \ell^{\infty}(\mathbb{Z},\mathbb{R})$. Assume that, for each $E \in \sigma_{\rm ess}(H)$, one has $\displaystyle\lim_{k\to\infty}\|P_k(E)^{-1}\|=0$ and  $\displaystyle\lim_{k\to\infty} \det P_k(E) =\infty$. Let $\eta>0$, \(a>3\),\,\(\tau > 4 \), and let $b \in \ell^{\infty}(\mathbb{Z},\mathbb{R})$ such that
\[
|b(n)| \leq \eta (a+2||V||_\infty)^{-\tau |n|}, \, \,  n \in \mathbb{Z}.
\]
 Let $|\kappa| \leq  {1}/{||b||_\infty}$ and let \( H^\kappa = \triangle +V + \kappa b \). Then, there exist universal constants $\gamma, \gamma'>0$ such that, for every $k \geq 1$ and every $E \in \sigma_{\rm ess}(H^\kappa)$,
\begin{enumerate}
\item[(i)] $\gamma^{-1} ||P_k(E)|| \leq  ||P_k^\kappa(E)|| \leq \gamma ||P_k(E)||,$
\item[(ii)] ${\gamma'}^{-1} \det P_k(E) \leq  \det P_k^\kappa(E) \leq \gamma' \det P_k(E)$, 
\end{enumerate}
where $P_k(E)$ and $P^\kappa_k(E)$ stand, respectively, for the operators defined in~\eqref{operators} for $H$ and $H^\kappa$.
\end{theorem}

 
\subsection{Proof of Theorem \ref{app3}}

\noindent $(i)$. It is a direct consequence of Theorem 1.1 in \cite{Damanik4}    
\

\noindent $(ii)$. Let $E \in \sigma_{\rm ess}(H^\kappa)$. It follows from Corollary 4.7 in \cite{AvilaUaH} that there exists a constant \( \tilde{\gamma} > 0 \) such that for each $k\geq 1$, 
\[
\|P_k(E)\| \leq \tilde{\gamma} \delta_k^{-\frac{3}{2}},
\]
where \( \delta_k = \dfrac{1}{\sqrt{4 \det P_k(E)}} \). We note that $P_{k}(E)$ is an increasing family of positive self-adjoint linear maps; $\det P_{k}(E)$ and $||P_{k}(E)||$ are increasing positive functions with $\displaystyle\lim_{k \to \infty} \det P_{k}(E) = \infty$ and  $\displaystyle\lim_{k \to \infty} ||P_{k}(E)|| = \infty$. Moreover, $\displaystyle\lim_{k\to\infty}\|P_{k}^{-1}(E)\|=0$ \cite{AvilaUaH}. Therefore, it follows from Lemma \ref{lemmaAvila} and Theorem~\ref{transferthm2} that for every \( E \in \sigma_{\rm ess}(H^\kappa) \),
\begin{equation}\label{MIMP}
\psi(m_\kappa^{+}(E + i\epsilon_k)) \leq  \gamma' \epsilon_k^{-\frac{1}{2}}, 
\end{equation}
where $\gamma':=\gamma\tilde{\gamma}>0$, \( \epsilon_k = \dfrac{1}{\sqrt{4 \det P_k^\kappa(E)}} \), and \( m_\kappa^+ \) stands for the \( m_\kappa \)-function associated  with the operator \( H_+^\kappa \). It follows from \eqref{maineqeq1}, \eqref{maineqeq2} and~\eqref{MIMP} that   
\begin{equation}\label{eq078}
( \gamma')^{-1} \epsilon_k^{\frac{1}{2}} \leq  \Im M_\kappa(E + i\epsilon_k) \leq  \gamma' \epsilon_k^{-\frac{1}{2}}.    
\end{equation}

Now, note that, for any bounded potential and any solution \( u \), we have \( \|u\|_{L+1} \leq C \|u\|_L \). Thus, by equation \eqref{detPk}, \( \epsilon_{k+1} > c \epsilon_k \), for some positive constants \( C, c > 0 \). Note also that the following function \( g(\epsilon) = \epsilon^{-1} \Im M_\kappa(E + i\epsilon) \) is monotonic.

Finally, let \( \epsilon \in (0, 1] \). Then there exists \( k \geq 1 \) such that $\epsilon_{k+1} \leq \epsilon \leq \epsilon_k.$ By the monotonicity of \( g(\epsilon) \), we have
\[
\frac{1}{\epsilon} \Im M_\kappa(E + i\epsilon)
\leq \frac{1}{\epsilon_{k+1}} \Im M_\kappa(E + i\epsilon_{k+1}),
\]
that is,
\[
\Im M_\kappa(E + i\epsilon)
\leq \frac{\epsilon}{\epsilon_{k+1}} \Im M_\kappa(E + i\epsilon_{k+1}).
\]
Using \eqref{eq078}, we have
\[
\Im M_\kappa(E + i\epsilon)
\leq \gamma' \epsilon_{k+1}^{-3/2}\epsilon.
\]
Since \( \epsilon_{k+1} \geq c \epsilon_k \geq c \epsilon \), it follows that $\epsilon_{k+1}^{-3/2} \leq c^{-3/2} \epsilon^{-3/2}.$
Thus, we conclude that
\[
\Im M_\kappa(E + i\epsilon)
\leq \gamma' c^{-3/2} \epsilon^{1 - 3/2}
=\gamma'' \epsilon^{-1/2},
\]
where \( \gamma'' := \gamma' c^{-3/2} \).
Therefore, for every \( \epsilon \in (0, 1]\),
\[
\Im M_\kappa(E + i\epsilon) \leq \gamma'' \epsilon^{-1/2}.
\]
By the same argument, we can obtain the lower bound. Thus, for every \( \epsilon \in (0,1] \),
\begin{equation*}\label{eq00000000000000001}
(\eta)^{-1} \epsilon^{\frac{1}{2}} \leq  \Im M_\kappa(E + i\epsilon) \leq  \eta \epsilon^{-\frac{1}{2}}, \, \eta>0,
\end{equation*}
and the result for the canonical spectral measure follows from \eqref{eq0000000000000000}, which implies the statement of the theorem, since for all \( f, g \in C(\sigma_{\rm ess}(H^\kappa)) \),
\begin{equation*}
    \mu_{f(H^\kappa)\delta_0+g(H^\kappa)\delta_1}  \leq 4 \max\{\left\| f\right\|_{\infty}^2,\left\| g\right\|_{\infty}^2\} \mu.  
\end{equation*}


\subsection{Proof of the Asymptotic formula II }\label{sec5}
\ 

\begin{lemma}\label{03teclemma}
Let $\{A_k\}_{k \geq 1}$ and $\{B_k\}_{k \geq 1}$ be sequences of $2 \times 2$ matrices such that
\[
A_k = C B_k D + R_k,
\]
where $\{B_k\}$, $C,$ and $ D$ are invertible matrices and $\{R_k\}$ is a uniformly bounded sequence (that is, there is $\gamma > 0$ so that $\|R_k\| \leq \gamma$ for each $k$). Assume that $\displaystyle\lim_{k\to\infty}\|B_k^{-1}\|=0$ and $\displaystyle\lim_{k\to\infty} \det B_k =\infty$. Then, for  sufficiently large $k$,
\[
\det(A_k) = \det(C)\det(D)\det(B_k)\left(1 + o(1)\right);
\]
in particular, there exists $\gamma'>0$ such that for each $k \geq 1$,
\[
 |\det(A_k)| \leq \gamma' ||C||^2 ||D||^2 |\det(B_k)|.
\]
\end{lemma}

\begin{proof}
We begin by rewriting $A_k$ as
\[
A_k = C B_k D + R_k = C \left(B_k + C^{-1} R_k D^{-1} \right) D.
\]
Set, for each $k \geq 1$,
\[
Q_k := C^{-1} R_k D^{-1}.
\]
Since $R_k$ is uniformly bounded and $C$, $D$ are fixed and invertible, the sequence $\{Q_k\}$ is also uniformly bounded. Moreover, since $\|B_k^{-1}\| \to 0$,  then 
$\|B_k^{-1} Q_k\| \to 0$ as $k \to \infty$.

On the other hand, one gets
\[
\det(A_k) = \det(C) \cdot \det(B_k) \cdot \det\left(I + B_k^{-1} Q_k\right) \cdot \det(D).
\]

In order to estimate $\det(I + B_k^{-1} Q_k)$, we use the following expansion of the determinant; for all $k$,
\[
\det\left(I + B_k^{-1} Q_k\right) = 1 + \det(B_k^{-1} Q_k) + \operatorname{Tr}(B_k^{-1} Q_k), 
\]
so 
\[
\det\left(I + B_k^{-1} Q_k\right) = 1 + o(1), \, k \to \infty.
\]

Thus, for all sufficiently large $k$,
\[
\det(A_k) = \det(C) \cdot \det(D) \cdot \det(B_k) \cdot (1 + o(1)),
\]
and we are done.
\end{proof}

We proceed with the proof of Theorem \ref{transferthm2}. First, we prove part $(i)$. By one of the identities stated in Theorem \ref{transferthm}, we write, for each $E \in \sigma_{\rm ess}(H^\kappa)$ and each \( n \geq 1 \),
\begin{eqnarray*} 
T^\kappa(E, 2n-1, 0) &=& T(E, 2n-1, 0)(I - Q(E)) - R_{2n-1}(E), \\
(T^\kappa)^*(E, 2n-1, 0) &=& (I - Q(E))^* T^*(E, 2n-1, 0) - R_{2n-1}^*(E),
\end{eqnarray*}
with \( \|R_{2n-1}\| \leq \eta |\kappa| (a + 2 \|V\|_\infty)^{(3 - \tau) (2n-1)} \). Therefore, by Lemma \ref{02teclemma},
\begin{eqnarray*} 
&\,&(T^\kappa)^*(E, 2n-1, 0) T^\kappa(E, 2n-1, 0)\\ &=& (I - Q(E))^* T^*(E, 2n-1, 0) T(E, 2n-1, 0)(I - Q(E)) \\
&& +\; R'_{2n-1}(E)
\end{eqnarray*}
with \( \|R'_{2n-1}\| \leq \eta' |\kappa| (a + 2 \|V\|_\infty)^{(4 - \tau) (2n-1)} \) for some $\eta'>0$. Thus,
\begin{eqnarray}\label{eqeqeqeq01010101} 
P_k^\kappa(E) = (I - Q(E))^* P_k(E)(I - Q(E)) + \sum_{n=1}^k R'_{2n-1}(E).
\end{eqnarray}
From the identity above, one of the inequalities in Theorem \ref{transferthm2} $(i)$ follows. The other follows similarly from the second identity in Theorem \ref{transferthm}.

Now, we prove part $(ii)$. By identity \eqref{eqeqeqeq01010101} and Lemma \ref{03teclemma}, it suffices to show that \( (I - Q(E)) \) is invertible and that \( \|(I - Q(E))^{-1}\| \leq c \) for some \( c > 0 \), independent of \( E \in \sigma_{\rm ess}(H^\kappa)\). Namely, let \( E \in \sigma(H^\kappa) \). Recall above that, for each \( n \geq 1 \),
\begin{eqnarray*}\nonumber 
&\,&(I - Q(E))\\ &=& T^{-1}(E, 2n-1, 0)T_\kappa(E, 2n-1, 0) + T^{-1}(E, 2n-1, 0)R_{2n-1} \\
&=& T^{-1}(E, 2n-1, 0)T_\kappa(E, 2n-1, 0)\left( I + T_\kappa^{-1}(E, 2n-1, 0)R_{2n-1} \right)
\end{eqnarray*}
with \( \|R_{2n-1}\| \leq \gamma (a + 2\|V\|_\infty)^{(3 - \tau)(2n - 1)} \)  and \( \|Q(E)\| \leq \gamma \) for some \( \gamma > 0 \), independent of \( E \in \sigma(H^\kappa) \).

Since, by Lemma \ref{02teclemma},
\begin{eqnarray*}
\|T_\kappa^{-1}(E, 2n-1, 0)R_{2n-1}\| &\leq& \|T_\kappa^{-1}(E, 2n-1, 0)\| \cdot \|R_{2n-1}\| \\
&=& \|T_\kappa(E, 2n-1, 0)\| \cdot \|R_{2n-1}\| \\
&\leq& \gamma' (a + 2\|V\|_\infty)^{(4 - \tau)(2n-1)}, \, \,  \gamma' > 0, \, \tau > 4,
\end{eqnarray*}
then for all \( n > n_0 \), by the Neumann series, \( \left( I + T_\kappa^{-1}(E, 2n-1, 0)R_{2n-1} \right) \) is invertible and  
\[
\left\| \left( I + T_\kappa^{-1}(E, 2n-1, 0)R_{2n-1} \right)^{-1} \right\| < 2.
\]
This implies that \( (I - Q(E)) \) is invertible. Finally, for all \( n > n_0 \),
\begin{eqnarray*}\nonumber 
&\,&\left\| \left( T^{-1}(E, 2n-1, 0)T_\kappa(E, 2n-1, 0) \right)^{-1} \right\| \\
&=& \left\|T^{-1}(E, 2n-1, 0)T_\kappa(E, 2n-1, 0) \right\| \\
&\leq& \|(I - Q(E))\| \cdot \left\| \left( I + T_\kappa^{-1}(E, 2n-1, 0)R_{2n-1} \right)^{-1} \right\| \\
&\leq& 2(1 + \gamma),
\end{eqnarray*}
which implies that \( \|(I - Q(E))^{-1}\| \leq 4(1 + \gamma) \). This concludes the proof of the theorem.
\hfill \qedsymbol


\appendix

\section{Auxiliary results}\label{appendix}

This appendix contains the proofs of auxiliary results used throughout the text.

\begin{lemma}\label{lemmaderrando}
Let $\triangle_\beta$ be the discrete Laplacian on $\ell^2(\mathbb{Z}^+)$ with boundary condition
\[
u(0) \cos(\beta) + u(1) \sin(\beta) = 0, \, \,  \beta \in \left(-\frac{\pi}{2}, \frac{\pi}{2}\right).
\]
The Radon-Nikodym derivative of the spectral measure $\mu_{\delta_1}^\beta$ is:
\[
\rho_{\delta_1}^\beta(x) = \frac{\cos^2(\beta) \, \sqrt{4 - x^2}}{\pi \left(2 + \sin(2\beta) \, x \right)}, \, \,  x \in (-2, 2).
\]
\end{lemma}

\begin{proof}
The outline of this proof is an exercise in \cite{Carmona2}, page 171. Let's compute 
\[
u = (\triangle_\beta - z)^{-1} \delta_1
\]
to obtain the Green's function; that is, we want to find \( u \in \ell^2(\mathbb{Z}^+) \) such that, for every \( n \geq 1 \),
\begin{equation}\label{eqderec}
u(n+1) + u(n-1) - z u(n) = \delta_{n1},
\end{equation}
with the boundary condition
\[
u(0) \cos(\beta) + u(1) \sin(\beta) = 0.
\]

For \( n \geq 2 \), the equation \eqref{eqderec} is
\[
u(n+1) + u(n-1) - z u(n) = 0
\]
with general solution
\[
u(n) = A e^{-i n \theta} + B e^{i n \theta}, \, \,  \text{with } z = 2 \cos(\theta).
\]
To ensure that \( u \in \ell^2(\mathbb{Z}^+) \) and \( \operatorname{Im} z > 0 \), we require \( \operatorname{Im} \theta < 0 \), which implies \( B = 0 \). Thus, we found the following solution for \eqref{eqderec}
\[
u(n) = A e^{-i n \theta}, \, \,  n \geq 1.
\]

We apply the boundary condition to obtain \( u(0) \), and then
\[
A = -\frac{\cos(\beta)}{\cos(\beta) + e^{-i\theta} \sin(\beta)}.
\]

Through the identities above, we obtain the Green's function
\[
G_\beta(z) = u(1) = A e^{-i \theta} = -\frac{\cos(\beta)\, e^{-i\theta}}{\cos(\beta) + e^{-i\theta} \sin(\beta)}.    
\]

Equivalently,
\begin{equation}\label{grenf}
G_\beta(z) = -\frac{\cos(\beta)}{\sin(\beta) + e^{i\theta} \cos(\beta)}.
\end{equation}

Finally, for \( x \in (-2,2) \), we write \( z = x + i\varepsilon \), compute \( \operatorname{Im} G_\beta(x + i\varepsilon) \), and take the limit \( \varepsilon \to 0^+ \). Thus, for every \( x = 2\cos(\theta) \) and \( \theta \in (0,\pi) \),
\[
\operatorname{Im} G_\beta(x + i0^+) = \cos^2(\beta) \cdot \frac{\sin(\theta)}{1 + \sin(2\beta)\cos(\theta)}.
\]
Substituting \( x = 2\cos(\theta) \), we get
\[
\sin(\theta) = \sqrt{1 - (x/2)^2} = \frac{\sqrt{4 - x^2}}{2}.
\]
Hence, by Stone's formula,
\[
\rho_{\delta_1}^\beta(x) = \frac{\cos^2(\beta) \cdot \sqrt{4 - x^2}}{\pi(2 + \sin(2\beta) x)}.
\]
\end{proof}

\noindent {\bf Proof of Lemma \ref{teclemmamainthm}}

The eigenvalues of $A$ are complex conjugate numbers of modulus 1:
\[
\lambda_{\pm} = e^{\pm i \theta},
\]
with $\theta = \arccos\left( \frac{x}{2} \right) \in (0, \pi)$. The corresponding eigenvectors $v_{\pm}$ satisfy
\[
A v_{\pm} = e^{\pm i \theta} v_{\pm}.
\]
A direct computation yields $A = P D P^{-1}$, where 
\[
P = \begin{pmatrix}
1 & 1 \\
e^{-i \theta} & e^{i \theta}
\end{pmatrix} \, \,  {\rm and} \, \,  D = \begin{pmatrix}
e^{i \theta} & 0 \\
0 & e^{-i \theta}
\end{pmatrix}.
\]
Now, 
\[
P^{-1} (\cos(\beta), -\sin(\beta))^{\mathrm T} = \frac{1}{2i \sin(\theta)}
\begin{pmatrix}
e^{i\theta} \cos(\beta) + \sin(\beta) \\
 -e^{-i\theta} \cos(\beta) - \sin(\beta)
\end{pmatrix},
\]
from which it follows that
\begin{eqnarray*}
&\,& \left\| P^{-1} (\cos(\beta), -\sin(\beta))^{\mathrm T} \right\|^2\\ &=& \left( \frac{1}{2 \sin(\theta)} \right)^2 \left( \left| e^{i\theta} \cos(\beta) + \sin(\beta) \right|^2 + \left| -e^{-i\theta} \cos(\beta) - \sin(\beta) \right|^2 \right) \\
&=& \frac{1}{4 \sin^2(\theta)} \cdot 2 \left| e^{i\theta} \cos(\beta) + \sin(\beta) \right|^2 \\
&=& \frac{1}{2 \sin^2(\theta)} \left[ (\cos(\theta) \cos(\beta) + \sin(\beta))^2 + (\sin(\theta) \cos(\beta))^2 \right] \\
&=& \frac{1}{2 \sin^2(\theta)} [ \cos^2(\theta) \cos^2(\beta)\\
&+& 2 \cos(\theta) \cos(\beta) \sin(\beta) + \sin^2(\beta) + \sin^2(\theta) \cos^2(\beta) ] \\
&=& \frac{1}{2 \sin^2(\theta)} [ (\cos^2(\theta) + \sin^2(\theta)) \cos^2(\beta)\\
&+& 2 \cos(\theta) \cos(\beta) \sin(\beta) + \sin^2(\beta) ] \\
&=& \frac{1}{2 \sin^2(\theta)} \left[ \cos^2(\beta) + \sin^2(\beta) + 2 \cos(\theta) \cos(\beta) \sin(\beta) \right] \\
&=& \frac{1 + \sin(2\beta) \cos(\theta)}{2 \sin^2(\theta)}.
\end{eqnarray*}
Finally, since \( D \) is a unitary operator, for every $n \geq 1$, one has
\[
\left\| D^n P^{-1} (\cos(\beta), -\sin(\beta))^{\mathrm T} \right\|^2 = \frac{1 + \sin(2\beta) \cos(\theta)}{2 \sin^2(\theta)}.
\]
\hfill \qedsymbol

\ 

\noindent {\bf Proof of Lemma \ref{01teclemma}}

The result follows by induction. For \(n = 1\), the left-hand side is
\[
F_1F_0 - G_1G_0.
\]
The right-hand side is
\[
\sum_{k=0}^1 \left( \prod_{j=k+1}^1 F_j \cdot (F_k - G_k) \cdot \prod_{j=0}^{k-1} G_j  \right).
\]
For \(k = 0\), we have
\[
 \prod_{j=1}^1 F_j = F_1 \, \,  \text{and} \, \,  \prod_{j=0}^{-1} G_j = 1 \, \,  \text{(empty product)}.
\]
Thus, the term for \(k = 0\) is
\[
F_1  \cdot (F_0 - G_0) \cdot 1= F_1 F_0 - F_1 G_0.
\]
The term for \(k = 1\) is
\[
(F_1 - G_1)  \cdot G_0  = F_1 G_0 - G_1 G_0.
\]
Hence, for \(n = 1\), the formula holds.

Now, assume that the formula holds for \(n\). For $n+1$, the left-hand side is
\begin{eqnarray*}
\prod_{j=0}^{n+1} F_j - \prod_{j=0}^{n+1} G_j &=& \left( F_{n+1} \prod_{j=0}^n F_j \right) - \left(G_{n+1} \cdot \prod_{j=0}^n G_j \right) \\ &=&  F_{n+1} \cdot \left(\prod_{j=0}^n F_j - \prod_{j=0}^n G_j\right)  + (F_{n+1} - G_{n+1}) \cdot \prod_{j=0}^n G_j.    
\end{eqnarray*}
Using the induction hypothesis, one has
\begin{eqnarray*}
\prod_{j=0}^{n+1} F_j - \prod_{j=0}^{n+1} G_j &=& F_{n+1} \cdot \left[\sum_{k=0}^n \left( \prod_{j=k+1}^n F_j \cdot (F_k - G_k) \cdot \prod_{j=0}^{k-1} G_j  \right)\right]\\   &+& (F_{n+1} - G_{n+1}) \cdot \prod_{j=0}^n G_j .    
\end{eqnarray*}
Including \(F_{n+1}\) in the product $\prod_{j=k+1}^n F_j$ and including the additional term for \(k = n+1\), this becomes
\[
\prod_{j=0}^{n+1} F_j - \prod_{j=0}^{n+1} G_j = \sum_{k=0}^{n+1} \left( \prod_{j=k+1}^{n+1} F_j \cdot (F_k - G_k) \cdot \prod_{j=0}^{k-1} G_j  \right).
\]
\hfill \qedsymbol


\begin{center} \Large{Acknowledgments} 
\end{center}

\noindent  M. Aloisio thanks the partial support by FAPEMIG (Minas Gerais state agency; under contract 01/24/APQ-03132-24) and C. R. de Oliveira thanks the partial support by CNPq (a Brazilian government agency, under contract 303689/2021-8). The authors are grateful to Grzegorz \'Swiderski for sharing the references \cite{kreimerLS2009,lukic2019}. Also to Zhenfu Wang for reading this version of the paper and for their valuable and constructive suggestions.


\end{document}